\documentclass[a4paper,10pt]{article}

\usepackage{geometry}
\usepackage{epsfig}
\usepackage{amsmath}
\usepackage{amssymb}
\usepackage{amsthm}
\usepackage{mathrsfs}
\usepackage{color}
\usepackage{amscd}

\usepackage{amsthm}
\newtheorem{theorem}{Theorem}[section]
\newtheorem{definition}[theorem]{Definition}
\newtheorem{lemma}[theorem]{Lemma}
\newtheorem{proposition}[theorem]{Proposition}
\newtheorem{corollary}[theorem]{Corollary}
\newtheorem{remark}[theorem]{Remark}
\newtheorem{example}[theorem]{Example}

\usepackage{amsfonts}

\newcommand{\oo}{{\mathbb{O}}}
\newcommand{\hh}{{\mathbb{H}}}

\newcommand{\cc}{{\mathbb{C}}}
\newcommand{\rr}{{\mathbb{R}}}
\newcommand{\zz}{{\mathbb{Z}}}
\newcommand{\nn}{{\mathbb{N}}}

\newcommand{\s}{{\mathbb{S}}}

\newcommand\vs[1]{{#1}_s^\circ}

\newcommand{\punto}{\bullet}

\newcommand{\sto}{\mathrm{u}}
\newcommand{\Sto}{\mathrm{U}}
\newcommand\re{\operatorname{Re}}
\newcommand\im{\operatorname{Im}}

\newcommand{\OO}{\Omega}

\newcommand{\ord}{\mathrm{ord}}

\newcommand{\mr}{\mathrm}
\newcommand{\mscr}{\mathscr}
\newcommand{\mc}{\mathcal}
\newcommand{\hslashslash}{%
  \raisebox{.9ex}{%
    \scalebox{.7}{%
      \rotatebox[origin=c]{18}{$-$}%
    }%
  }%
}
\newcommand{\fslash}{%
  {%
   \vphantom{f}%
   \ooalign{\kern.05em\smash{\hslashslash}\hidewidth\cr$f$\cr}%
   \kern.05em
  }%
}

\title{\bf Division algebras of slice functions}

\author{Riccardo Ghiloni\\
 Alessandro Perotti\\
\small Dipartimento di Matematica, Universit\`a di Trento\\ 
\small Via Sommarive 14, I-38123 Povo Trento, Italy\\
\small riccardo.ghiloni@unitn.it, alessandro.perotti@unitn.it\\
\and
Caterina Stoppato
\\ 
\small Dipartimento di Matematica e Informatica ``U. Dini'', Universit\`a di Firenze \\
\small Viale Morgagni 67/A, I-50134 Firenze, Italy\\
\small stoppato@math.unifi.it}

\date{  }

\begin{document}

\maketitle


\begin{abstract}
This work studies slice functions over finite-dimensional division algebras. Their zero sets are studied in detail along with their multiplicative inverses, for which some unexpected phenomena are discovered. The results are applied to prove some useful properties of the subclass of slice regular functions, previously known only over quaternions. Firstly, they are applied to derive from the maximum modulus principle a version of the minimum modulus principle, which is in turn applied to prove the open mapping theorem. Secondly, they are applied to prove, in the context of the classification of singularities, the counterpart of the Casorati-Weierstrass theorem.
\end{abstract}


\section{Introduction}\label{sec:introduction}

This work addresses the study of function theory over finite-dimensional division algebras with a unified vision, thanks to the theory of slice functions introduced in~\cite{perotti}.

As explained in~\cite{survey}, complex holomorphy admits a natural generalization to such algebras: the notion of \emph{slice regular function} introduced in~\cite{cras,advances} for the algebra of quaternions and in~\cite{rocky} for the algebra of octonions. The class of slice regular functions includes polynomials and convergent power series of the form
\[f(x)=\sum_{n\in\nn}x^na_n\,,\]
and it has many useful properties. Quaternionic slice regular functions have been extensively studied: see~\cite{librospringer} for a survey of the first phase of their study. Over the octonions, power series have been investigated in~\cite{ghiloni}, while slice regular functions have been considered in the recent work~\cite{wang}. A key tool for these studies was a quaternionic result called representation formula,~\cite[theorem 2.27]{cauchy}, along with the related octonionic result~\cite[lemma 1]{ghiloni}.

The work~\cite{perotti} introduced an innovative approach. On the one hand, it provided a definition of slice regularity valid for functions with values in any alternative $^*$-algebra $A$. On the other hand, it widened the class under investigation by relaxing the regularity assumptions on the functions. Indeed, it defined the class of \emph{slice functions} to comprise exactly those $A$-valued functions (not necessarily differentiable nor continuous) for which the analog of the representation formula is valid. This class properly includes the class of $A$-valued slice regular functions and it provides new tools to study it. For instance, the algebraic properties of slice functions, studied in~\cite{gpsalgebra}, have been applied in~\cite{gpssingularities} to the construction of Laurent series and the classification of singularities of slice regular functions.

In the present work, we turn back to the case of finite-dimensional division algebras. This case affords a richer algebraic structure, which allows us to deepen the study of slice functions. In doing so, we encounter both expected and unexpected phenomena. In the second part of the work, we derive some results valid for slice regular functions over finite-dimensional division algebras, which had not been proven with the original approach to slice regularity.

In Section~\ref{sec:preliminaries}, we recall some facts about finite-dimensional division algebras and some properties of slice functions and slice regular functions.

The grounds for our work thus set, we proceed in Section~\ref{sec:zeros} to a detailed description of the zero sets of slice functions over finite-dimensional division algebras. Their peculiar properties are direct extensions of those of quaternionic slice regular functions,~\cite{advances,zeros,advancesrevised,zerosopen,altavillawithoutreal}, and of octonionic power series,~\cite{rocky,ghiloni}.

Section~\ref{sec:reciprocal} is devoted to reciprocals (multiplicative inverses) of slice functions over finite-dimensional division algebras. We present a representation formula for the reciprocal, which is brand new both over quaternions and over octonions. We generalize to all division algebras a formula of~\cite{poli,zerosopen,gpsalgebra}, which linked the values of a quaternionic slice function to those of its reciprocal. This generalization is highly nontrivial and some unexpected topological phenomena appear in the octonionic case.

In the last part of the work, we focus on slice regular functions over finite-dimensional division algebras. In Section~\ref{sec:open}, we prove the maximum modulus principle and then apply the results about reciprocals to derive a version of the minimum modulus principle. The latter principle is, in turn, applied to prove the open mapping theorem. These results subsume the separate results of~\cite{advances,open,zerosopen,altavillawithoutreal,wang} and they strengthen them both over quaternions and over octonions.

Section~\ref{sec:singularities} studies the possible singularities of slice regular functions. We make use of the general theory of~\cite{gpssingularities} to classify them as removable singularities, poles or essential singularities, but we add a finer characterization of the essential singularities that is typical of division algebras. We then have two results proven in~\cite{poli,zerosopen} for quaternions, but new for octonions. The first one is the counterpart of the Casorati-Weierstrass theorem. The second one concerns the analogs of meromorphic functions, called \emph{semiregular functions}. We prove that semiregular functions form an (infinite-dimensional) division algebra.


\section{Preliminaries}\label{sec:preliminaries}

Let $\cc,\hh,\oo$ denote the $^*$-algebras of complex numbers, quaternions and octonions, respectively. As explained in~\cite{ebbinghaus,libroward,baez}, they can be built from the real field $\rr$ by means of the so-called Cayley-Dickson construction. For the octonions, it takes the form:
\begin{center}
$\oo=\hh+\ell\hh$, $(\alpha+\ell \beta)(\gamma+\ell \delta)=\alpha\gamma-\delta\beta^c+\ell(\alpha^c\delta+\gamma\beta)$, $(\alpha+\ell \beta)^c=\alpha^c-\ell \beta\quad\forall\,\alpha,\beta,\gamma,\delta\in\hh$.
\end{center}

On the one hand, this construction endows the three real vector spaces with a bilinear multiplicative operation, which makes each of them an \emph{algebra}. By construction, each of them is \emph{unitary}, that is, it has a multiplicative neutral element $1$; and $\rr$ is identified with the subalgebra generated by $1$. All algebras and subalgebras we consider are assumed to be unitary. It is well-known that $\oo$ is not commutative nor associative but it is \emph{alternative}: the associator $(x,y,z)=(xy)z-x(yz)$ of three elements vanishes whenever two of them coincide. The \emph{nucleus} $\mathfrak{N}(\oo):=\{r \in \oo\, |\, (r,x,y)=0\ \forall\, x,y \in \oo\}$ and the \emph{centre} $\{r \in \mathfrak{N}(\oo)\, |\, rx=xr\ \forall\, x\in \oo\}$ of the algebra $\oo$ both coincide with $\rr$.

On the other hand, the Cayley-Dickson construction endows each of $\cc,\hh,\oo$ with a \emph{$^*$-involution}, i.e., a (real) linear transformation $x\mapsto x^c$ with the following properties: $(x^c)^c=x$ and $(xy)^c=y^cx^c$ for every $x,y$; $x^c=x$ for every $x \in \rr$. Thus, $\cc,\hh$ and $\oo$ are \emph{$^*$-algebras}. We point out that $(r+v)^c=r-v$ for all $r \in \rr$ and all $v$ in the Euclidean orthogonal complement of $\rr$. The \emph{trace} and \emph{norm} functions, defined by the formulas
\begin{equation}\label{traceandnorm}
t(x):=x+x^c \quad \text{and} \quad n(x):=xx^c,
\end{equation}
are real-valued. In particular, $\oo$ is \emph{compatible}, i.e., its trace function $t$ has values in the nucleus of the algebra. This definition has been given in~\cite[\S1]{gpsalgebra}, along with the following property.
\begin{itemize}
\item {[$^*$-Artin's theorem]} In a compatible $^*$-algebra, the $^*$-subalgebra generated by any two elements is associative.
\end{itemize}
Moreover, in $\cc,\hh$ and $\oo$, $\frac{1}{2}t(xy^c)$ coincides with the standard scalar product $\langle x,y\rangle$ and $n(x)$ coincides with the squared Euclidean norm $\Vert x\Vert^2$. In particular, these algebras are \emph{nonsingular}, i.e., $n(x)=xx^c=0$ implies $x=0$. 
The trace function $t$ vanishes on every commutator $[x,y]:=xy-yx$ and on any associator $(x,y,z)$ (see~\cite[lemma 5.6]{gpsalgebra}). It holds $n(xy) = n(x) n(y)$, or equivalently $\Vert xy\Vert = \Vert x\Vert\,\Vert y\Vert$.

Every nonzero element $x$ of $\cc$, $\hh$ or $\oo$ has a multiplicative inverse, namely $x^{-1} = n(x)^{-1} x^c = x^c\, n(x)^{-1}$. For all elements $x,y$: if $x\neq0$ then $(x^{-1},x,y)=0$; if $x,y\neq 0$ then $(xy)^{-1} = y^{-1}x^{-1}$.
As a consequence, each of the algebras $\cc,\hh,\oo$ is a division algebra and has no zero divisors. 
A famous result due to Zorn states that $\rr,\cc,\hh$ and $\oo$ are the only (finite-dimensional) alternative division algebras.

From this point on, \emph{let $A$ be any of the algebras $\cc,\hh,\oo$}. Let us consider the sphere of \emph{imaginary units}
\begin{equation} \label{eq:s_A}
\s=\s_A:=\{x \in A\, | \, t(x)=0, n(x)=1\} = \{w \in A \, |\, w^2=-1\}\,,
\end{equation}
which has, respectively, dimension $0,2$ or $6$.
The $^*$-subalgebra generated by any $J \in \s$, i.e., $\cc_J=\langle 1,J \rangle$, is $^*$-isomorphic to the complex field $\cc$ (endowed with the standard multiplication and conjugation) through the $^*$-isomorphism 
\[\phi_J\ :\ \cc\to\cc_J,\quad \alpha+i\beta\mapsto\alpha+\beta J\,.\]
The union
\begin{equation} \label{eq:slice}
\text{$\bigcup_{J \in \s}\cc_J$}
\end{equation}
coincides with the entire algebra $A$. If, moreover, $A=\hh,\oo$, then $\cc_I \cap \cc_J=\rr$ for every $I,J \in \s$ with $I \neq \pm J$. As a consequence, every element $x$ of $A \setminus \rr$ can be written as follows: $x=\alpha+\beta J$, where $\alpha \in \rr$ is uniquely determined by $x$, while $\beta \in \rr$ and $J \in \s$ are uniquely determined by $x$, but only up to sign. If $x \in \rr$, then $\alpha=x$, $\beta=0$ and $J$ can be chosen arbitrarily in $\s$. Therefore, it makes sense to define the \emph{real part} $\re(x)$ and the \emph{imaginary part} $\im(x)$ by setting $\re(x):=t(x)/2=(x+x^c)/2 = \alpha$ and $\im(x):=x-\re(x)=(x-x^c)/2=\beta J$. It also makes sense to call the Euclidean norm $\Vert x\Vert=\sqrt{n(x)}=\sqrt{a^2+\beta^2}$ the \emph{modulus} of $x$ and to denote it as $|x|$. The algebra $\oo$ has the following useful property.
\begin{itemize}
\item {[Splitting property]} For each imaginary unit $J \in \s$, there exist $J_1,J_2,J_3 \in \oo$ such that $\{1,J,J_1,JJ_1,J_2,JJ_2,J_3,JJ_3\}$ is a real vector basis of $\oo$, called a \emph{splitting basis} of $\oo$ associated to $J$. Moreover, $J_1,J_2,J_3$ can be chosen to be imaginary units and to make the basis orthonormal.
\end{itemize}
An analogous property holds for $\hh$.

We consider on $A$ the natural Euclidean topology and differential structure as a finite-dimensional real vector space. The relative topology on each $\cc_J$ with $J \in \s$ clearly agrees with the topology determined by the natural identification between $\cc_J$ and $\cc$. Given a subset $E$ of $\cc$, its \emph{circularization} $\OO_E$ is defined as the following subset of $A$:
\[
\OO_E:=\left\{x \in A \, \big| \, \exists \alpha,\beta \in \rr, \exists J \in \s \mathrm{\ s.t.\ } x=\alpha+\beta J, \alpha+\beta i \in E\right \}\,.
\]
A subset of $A$ is termed \emph{circular} if it equals $\OO_E$ for some $E\subseteq\cc$. For instance, given $x=\alpha+\beta J \in A$ we have that
\[
\s_x:=\alpha+\beta \, \s=\{\alpha+\beta I \in A \, | \, I \in \s\}
\]
is circular, as it is the circularization of the singleton $\{\alpha+i\beta\}\subseteq \cc$. We observe that $\s_x=\{x\}$ if $x \in \rr$. On the other hand, for $x \in A \setminus \rr$, the set $\s_x$ is obtained by real translation and dilation from the sphere $\s$. Let $D$ be a non-empty subset of $\cc$ that is invariant under the complex conjugation $z=\alpha+i\beta \mapsto \overline{z}=\alpha-i\beta$: then for each $J \in \s$ the map $\phi_J$ naturally embeds $D$ into a ``slice'' $\phi_J(D)$ of $\OO_D=\bigcup_{J\in\s}\phi_J(D)$.
If, moreover, $D$ is a connected open subset of $\cc$ and it intersects the real line $\rr$, then $\OO_D$ is called a \emph{slice domain}. If, instead, an open $D$ does not intersect $\rr$ and it has two connected components switched by complex conjugation, then $\OO_D$ is called a \emph{product domain}.

For a given $D\subseteq\cc$ (invariant under the complex conjugation) and for $\OO=\OO_D$, we will work with the class $\mc{S}(\OO)$ of \emph{slice functions} $f:\OO\to A$, as defined in~\cite{perotti}.
With the operations of \emph{slice multiplication} and \emph{slice conjugation} defined in the same article, the algebraic structure of slice functions can be described as follows (see~\cite[\S 2]{gpsalgebra}).

\begin{proposition}\label{prop:algebra}
The slice functions $\OO \to A$ form an alternative $^*$-algebra over $\rr$ with pointwise addition $(f,g) \mapsto f + g$, slice multiplication $(f,g) \mapsto f \cdot g$ and slice conjugation $f \mapsto f^c$ on $\mc{S}(\OO)$. The $^*$-algebra $\mc{S}(\OO)$ is compatible and its centre includes the $^*$-subalgebra $\mc{S}_\rr(\OO)$ of slice preserving functions.
\end{proposition}

We recall that a slice function $f$ is \emph{slice preserving} if it maps every ``slice'' $\OO_D\cap\cc_J$ into $\cc_J$ and if it has the Schwarz reflection property (the latter property being a consequence of the former one when $A=\hh,\oo$, but not when $A=\cc$).
If $f$ is slice preserving then $(f \cdot g)(x)=(g \cdot f)(x)=f(x)g(x)$ and $f^c(x)=f(x)$ but these equalities do not hold in general. 
The \emph{normal function} of $f$ in $\mc{S}(\OO)$ is defined as
\[
N(f)=f \cdot f^c.
\]
It coincides with $f^2$ if $f$ is slice preserving.
It is useful to define the slice function $\vs f:\OO \to A$, called \emph{spherical value} of $f$, and the slice function $f'_s:\OO \setminus \rr \to A$, called \emph{spherical derivative} of $f$, by setting
\[
\vs f(x):=\frac{1}{2}(f(x)+f(x^c))
\quad \text{and} \quad
f'_s(x):=\frac{1}{2}\im(x)^{-1}(f(x)-f(x^c)).
\] 
The works~\cite{advancesrevised,perotti} showed that these functions are constant on each sphere $\s_x\subseteq\OO \setminus \rr$. For all $x \in \OO \setminus \rr$, it holds
\[
f(x)=\vs  f(x) + (\im \cdot\, f'_s)(x) = \vs  f(x)+ \im(x)  f'_s(x)\,,
\]
where the function $\im$ is a slice preserving element of $\mc{S}(A)$. The function $f$ is slice preserving if, and only if, $\vs f$ and $f'_s$ are real-valued. Moreover, the spherical value and derivative are used in the next expressions of the operations on $\mc{S}(\OO)$, see~\cite{gpsalgebra}. 
For all $x \in \OO \setminus \rr$:

\begin{equation}\label{decomposedconjugate}
f^c(x) = \underbrace{\vs  f(x)^c}_{\vs{(f^c)}(x)}+ \im(x)   \underbrace{f'_s (x)^c}_{(f^c)'_s (x)}\,,
\end{equation}
\begin{equation} \label{decomposednormal}
N(f)(x) =\underbrace{n(\vs  f(x)) + \im(x)^2 n( f'_s(x))}_{\vs{N(f)}(x)} + \im(x)\, \underbrace{t\big(\vs  f(x)\, f'_s(x)^c\big)}_{{N(f)}'_s(x)}\,,
\end{equation}
\begin{equation}\label{decomposedproduct}
f\cdot g = \underbrace{\vs  f\,\vs  g+\im^2 f'_s\, g'_s}_{\vs{(f\cdot g)}}+\im\, \underbrace{(\vs  f\, g'_s+ f'_s\, \vs  g)}_{{(f\cdot g)}'_s}\,,
\end{equation}
while for all $x \in \OO \cap \rr$ we have $f(x)=\vs f(x)$, $f^c(x)=f(x)^c$, $N(f)(x)=n(f(x))$, and $(f \cdot g)(x) = f(x) g(x)$.

Some special subclasses of the algebra $\mc{S}(\OO)$ of slice functions have been singled out in~\cite{perotti}: the nested $^*$-subalgebras $\mc{S}^0(\OO), \mc{S}^1(\OO), \mc{S}^\omega(\OO), \mc{SR}(\OO)$, obtained by imposing continuity, continuous differentiability, real analyticity and holomorphy (in appropriate senses). For all of these $^*$-subalgebras except the first one, $\OO$ is assumed to be open (whence a disjoint union of slice domains and product domains). The elements of $\mc{SR}(\OO)$ are called \emph{slice regular} functions and they have been first introduced in~\cite{cras,advances} for $A=\hh$ and in~\cite{rocky} for $A=\oo$.

The relation between slice regularity and complex holomorphy is made explicit in the following result, called the `splitting lemma' and proven in~\cite{rocky,perotti} for $A=\oo$.
    
\begin{lemma}\label{splitting} 
Let $\{1,J,J_1,JJ_1,J_2,JJ_2,J_3,JJ_3\}$ be a splitting basis for $\oo$ and let $\OO$ be an open circular subset of $\oo$. For $f \in \mc{S}^1(\OO)$, let $f_0,\ldots,f_3 : \OO_J \to \cc_J$ be the $\mscr{C}^1$ functions such that $f_{|_{\OO_J}}=\sum_{n=0}^3 f_n J_{n}$, where $J_0:=1$. Then $f$ is slice regular if, and only if, for each $n \in \{0,1,2,3\}$, $f_n$ is holomorphic from $\OO_J$ to $\cc_J$, both equipped with the complex structure associated to left multiplication by $J$.
\end{lemma}

An analogous result holds for $A=\hh$, see~\cite{advances,advancesrevised}. Polynomials and power series are examples of slice regular functions, see~\cite[theorem 2.1]{advances} and~\cite[theorem 2.1]{rocky}.

\begin{proposition}
Every polynomial of the form $\sum_{m=0}^n x^ma_m = a_0+x a_1 + \ldots x^n a_n$ with coefficients $a_0, \ldots, a_n \in A$ is a slice regular function on $A$. Every power series of the form $\sum_{n \in \nn} x^n a_n$ converges in a ball $B(0, R) = \{x \in A\ |\ \Vert x \Vert<R\}$. If $R>0$, then the sum of the series is a slice regular function on $B(0,R)$.
\end{proposition}

Actually, $\mc{SR}(B(0,R))$ coincides with the $^*$-algebra of power series converging in $B(0, R)$ with the operations
\[\left(\sum_{n \in \nn} x^n a_n\right)\cdot\left(\sum_{n \in \nn} x^n b_n\right) = \sum_{n \in \nn} x^n \sum_{k=0}^na_kb_{n-k}\text{\quad and\quad}
\left(\sum_{n \in \nn} x^n a_n\right)^c= \sum_{n \in \nn} x^n a_n^c\,.\]
This is a consequence of~\cite[theorem 2.7]{advances} and~\cite[theorem 2.12]{rocky}. With the same operations, the polynomials over $A$ form a $^*$-subalgebra of $\mc{SR}(A)$.

\begin{example}\label{ex:Delta}
If we fix $y \in A$, the binomial $f(x) := x-y$ is slice regular on $A$. The normal function $N(f)(x) = (x-y) \cdot (x-y^c) = x^2-x(y+y^c)+yy^c$ coincides with the slice preserving quadratic polynomial
$\Delta_y(x):=x^2-xt(y)+n(y)$.
The zero set of $\Delta_y$ is equal to $\s_y$. Consequently, if $y' \in A$, then $\Delta_{y'}=\Delta_y$ if and only if $\s_{y'}=\s_y$.
\end{example}

\subsection*{Convention}
{\bf
Throughout the paper, all statements concerning $\mc{S}(\OO)$ and its subalgebras are valid for circular sets $\OO$ in $\cc,\hh$ or $\oo$. In Section~\ref{sec:singularities}, we set $A=\cc,\hh$ or $\oo$. All proofs are specialized to the octonionic case for the sake of simplicity, but they stay valid when either $\cc$ or $\hh$ is substituted for $\oo$. All examples are over $\oo$, but the displayed functions can be easily restricted to the subalgebras $\hh$ and $\cc$.
}


\section{Zeros of slice functions}\label{sec:zeros}

In the present section, we describe the zero sets 
\[
V(f):=\{x \in \OO \, | \, f(x)=0\}
\]
of slice functions $f\in \mc{S}(\OO)$. In the special case of octonionic power series, similar results had been obtained in~\cite{rocky,ghiloni}. For quaternionic slice regular functions, see~\cite{advances,zeros,advancesrevised,zerosopen,altavillawithoutreal}.
In the general setting of slice functions we use the theory developed in~\cite{gpsalgebra}, taking into account the specific properties of division algebras and the next result.
Let us recall that a slice function $f$ is termed \emph{tame} if $N(f)$ is slice preserving and $N(f)=N(f^c)$.

\begin{theorem}\label{thm:tame}
Every $f\in \mc{S}(\OO)$ is tame. As a consequence,
\begin{equation}\label{Nfg}
N(f \cdot g) = N(f) N(g) = N(g) N(f) = N(g\cdot f)
\end{equation}
for all $f,g\in \mc{S}(\OO)$.
\end{theorem}

 \begin{proof}
 Let $f\in \mc{S}(\OO)$. The function $N(f)$ is slice preserving because the quantities
 \[n(\vs  f(x)) + \im(x)^2 n( f'_s(x)),\quad t\big(\vs  f(x)\, f'_s(x)^c\big)\]
 appearing in Formula~\eqref{decomposednormal} are real numbers. Moreover, $N(f)=N(f^c)$ because the aforementioned quantities are unchanged when $f^c$ is substituted for $f$. This is a consequence of Formula~\eqref{decomposedconjugate} and of the equalities
 \[n(a)= \Vert a\Vert^2 = \Vert a^c\Vert^2 = n(a^c),\quad t(ab^c)=2\langle a,b\rangle=2\langle a^c,b^c\rangle = t(a^cb)\]
 valid for all $a,b\in\oo$. Thus, $f$ is tame. The second assertion then follows from~\cite[proposition 2.4]{gpsalgebra}.
 \end{proof}

We are now ready to describe the zeros of slice functions $f,g$ on a sphere $\s_x$ and their relation to the zeros of $f^c,N(f),f\cdot g$. 
Indeed, the next result can be derived from~\cite[corollary 4.7 (2)]{gpsalgebra} and from~\cite[proposition 5.9]{gpsalgebra}, using theorem~\ref{thm:tame}.

\begin{theorem}\label{zerosonsphere}
If $f \in \mc{S}(\OO)$, then for every $x \in \OO$ the sets $\s_x \cap V(f)$ and $\s_x \cap V(f^c)$ are both empty, both singletons, or both equal to $\s_x$. Moreover,
 \[
V(N(f))=\bigcup_{x \in V(f)}\s_x=\bigcup_{x \in V(f^c)}\s_x
 \]
Finally, for all $g \in \mc{S}(\OO)$,
\[\bigcup_{x\in V(f\cdot g)}\s_x =\bigcup_{x\in V(f)\cup V(g)}\s_x.\]
\end{theorem}

\begin{example}\label{ex:binomial}
Fix $y\in\oo$. If $f(x) := x-y$, whence $f^c(x) =x-y^c$ and $N(f)=\Delta_y(x)
$, then 
\[V(f)=\{y\}\,,\quad V(f^c)=\{y^c\}\,,\quad V(N(f))=\s_y\,.\] For all constant functions $g\equiv c$, we have $(f\cdot g)(x)= xc-yc$, whence $V(f\cdot g)$ is $\{y\}$ when $c \neq 0$ and it is $\oo$ when $c=0$.
\end{example}

The general picture is much more manifold than the previous example tells. The next result describes the `camshaft effect' (so called in~\cite{ghiloni}, which studied the special case of octonionic power series).
It can be proven using~\cite[theorem 5.5]{gpsalgebra} and theorem~\ref{zerosonsphere}.

\begin{theorem}\label{Camshaft}
Let $f,g\in\mc{S}(\OO)$. If $x\in\OO$ is real and $x\in V(f)\cup V(g)$, then $x\in V(f\cdot g)$. More generally, for all $x\in\OO$,
\begin{enumerate} 
\item If $\s_x\subseteq V(f)$ or $\s_x\subseteq V(g)$, then $\s_x\subseteq V(f\cdot g)$.
\end{enumerate}
If $x\in\OO\setminus\rr$ then:
\begin{enumerate}
\item[2.]
If $\s_x\cap V(f)$ is a singleton $\{y\}$ and $\s_x\cap V(g)=\emptyset$, then $\s_x\cap V(f\cdot g)=\{w\}$, with
\[ w=\big((y f'_s(x)) \vs g(x)-(y\im(y) f'_s(x)) g'_s(x)\big) \big( f'_s(x)\vs g(x)-(\im(y) f'_s(x)) g'_s(x) \big)^{-1}.\]
\item[3.] 
If $\s_x\cap V(f)=\emptyset$ and $\s_x\cap V(g)$ is a singleton $\{z\}$,  then $\s_x\cap V(f\cdot g)=\{w\}$, with
\[w=\big(\vs f(x)(z g'_s(x))- f'_s(x)(z\im(z) g'_s(x)) \big) \big( \vs  f(x) g'_s(x)- f'_s(x)(\im(z)  g'_s(x)) \big)^{-1}.\]   
\item[4.]
If $\s_x\cap V(f)=\{y\}$ and $\s_x\cap V(g)=\{z\}$ for some $y,z \in \s_x$, then one of the following holds:
\begin{enumerate}
\item  $\s_x\subseteq V(f\cdot g)$; or
\item  $\s_x\cap V(f\cdot g)=\{w\}$, with 
\[w=\big( n(x) f'_s(x) g'_s(x) - (y f'_s(x)) (z g'_s(x))\big) \big( (f\cdot g)'_s(x) \big)^{-1};\]
\end{enumerate}
depending on whether or not $(f\cdot g)'_s(x) =(y^c f'_s(x)) g'_s(x)-f'_s(x) (z g'_s(x))$ vanishes.
\end{enumerate}
\end{theorem}
We point out that in case {\it 4.(b)} the point $w$ can be equivalently computed by means of each of the formulas appearing in {\it 2.} and {\it 3.}
        
The zero sets of slice regular functions can be further characterized by means of~\cite[theorem 4.11]{gpsalgebra} and~\cite[corollary 4.17]{gpsalgebra}, as follows. For $J\in\s$, we will use the notations $\cc_J^+=\{\alpha+\beta J\in\cc_J\,|\, \beta>0\}$ and  $\OO_J^+:=\OO\cap\cc_J^+$.

\begin{theorem}\label{structure_zeros}
Assume that $\OO$ is a slice domain or a product domain and let $f \in \mc{SR}(\OO)$.
\begin{itemize}
\item If $f\not\equiv0$ then the intersection $V(f)\cap \cc_J^+$ is closed and discrete in $\OO_J$ for all $J\in\s$ with at most one exception $J_0$, for which it holds $f_{|_{\OO_{J_0}^+}} \equiv 0$.
\item If, moreover, $N(f)\not\equiv0$ then $V(f)$ is a union of isolated points or isolated spheres $\s_x$.
\end{itemize}
\end{theorem}

An example with $f\not\equiv0$ and $f_{|_{\cc_{J_0}^+}} \equiv 0 \equiv N(f)$ has been provided in~\cite[remark 12]{perotti}:

\begin{example}\label{1fin}
Fix $J_0\in\s$. Let $f(x)=1+\frac{\im(x)}{|\im(x)|}J_0$ for each $x\in \oo\setminus\rr$. Then $f$ is slice regular in $\oo\setminus\rr$ and $V(f)=\cc_{J_0}^+$. Moreover $N(f)\equiv0$.
\end{example}

Understanding when $N(f)\equiv 0$ implies $f\equiv 0$ is the same as characterizing the nonsingularity of $\mc{SR}(\OO)$, which can be done as follows.

\begin{proposition}\label{SRnonsingular}
The $^*$-algebra $\mc{SR}(\OO)$ is nonsingular if, and only, $\OO$ is a union of slice domains.
\end{proposition}

The previous proposition follows directly from~\cite[proposition 4.14]{gpsalgebra}. Similarly, the next result derives from~\cite[proposition 5.18]{gpsalgebra},~\cite[corollary 5.19]{gpsalgebra} and theorem~\ref{thm:tame}.

\begin{proposition}
Assume that $\OO$ is a slice domain or a product domain. Take $f\in\mc{SR}(\OO)$ with $N(f)\not\equiv0$. Then, for $g\in\mc{S}^0(\OO)$, $f\cdot g\equiv0$ or $g\cdot f\equiv0$ implies $g\equiv0$. In particular, if $\OO$ is a slice domain, then no element of $\mc{SR}(\OO)$ can be a zero divisor in $\mc{S}^0(\OO)$.
\end{proposition}


\section{Reciprocals of slice functions}\label{sec:reciprocal}

This section treats the multiplicative inverses $f^{-\punto}$ of slice functions. We will make use of the general theory of~\cite{gpsalgebra} but also prove two new formulas for $f^{-\punto}$: a representation formula and a formula that links the values of $f^{-\punto}$ to those of $f$. The representation formula for $f^{-\punto}$ is a completely new result. The second formula was only known for the quaternionic case, see~\cite{poli,zerosopen,gpsalgebra}.

Our first statement follows immediately from~\cite[proposition 2.4]{gpsalgebra}, from formula \eqref{decomposednormal} and from theorem~\ref{thm:tame}.

\begin{proposition}\label{reciprocal}
Let $f \in \mc{S}(\OO)$. If $\OO' := \OO \setminus V(N(f))$ is not empty, then $f$ admits a multiplicative inverse $f^{-\punto}$ in $\mc{S}(\OO')$, namely
\[
f^{-\punto}(x) = (N(f)^{-\punto} \cdot f^c) (x)= (N(f)(x))^{-1} f^c(x).
\]
If $x\in\OO'\cap(\rr\cup V(f'_s))$ then $f^{-\punto}(x)=f(x)^{-1}$. Furthermore, if $\OO'$ is open then $f^{-\punto}$ is slice regular if and only if $f$ is slice regular in $\OO'$.
\end{proposition}

We point out that here and in the rest of the paper the restriction $f_{|_{\OO'}}$ is denoted again as $f$ and $f^{-\punto}$ stands for $\left(f_{|_{\OO'}}\right)^{-\punto}$. Similarly, the product $f\cdot g$ of two slice functions $f,g$ whose domains of definition intersect in a smaller domain $\tilde \OO\neq \emptyset$ should be read as $f_{|_{\tilde\OO}} \cdot g_{|_{\tilde\OO}}$.

We remark that if $\OO=\OO_D$ is a slice domain and $f \in \mc{SR}(\OO)$ has $N(f)\not \equiv 0$, then $V(N(f))$ is the circularization of a closed and discrete subset of $D$.

\begin{example}\label{ex:DeltaInv}
For fixed $y \in \oo$ and $f(x) := x-y$, we have
\[
f^{-\punto}(x) = \Delta_y(x)^{-1} (x-y^c) = (x^2-xt(y)+n(y))^{-1}(x-y^c)
\]
in $\mc{SR}(\OO')$, where $\OO' = \oo \setminus \s_y$. In the sequel, for each $n \in \zz$ we will denote by $(x-y)^{\punto n}$ the $n^{\rm{th}}$-power of $f$ with respect to the slice product. For the power $(x-y)^{\punto (-n)}$ we might also use the notation $(x-y)^{-\punto n}$.
\end{example}

More in general, we will be using the notation $f(x)\cdot g(x)$ for $(f\cdot g)(x)$ and the notation $f(x)^{\punto n}$ for $f^{\punto n}(x)$ when they are unambiguous. In such a case, $x$ will necessarily stand for a variable.

Now let us see how $f^{-\punto}$ can be represented in terms of $\vs f, f'_s$. To this end, the following definition will be useful.

\begin{definition}
For all $a,b$ with $n(a) \neq n(b)$ or $t(b^ca)\neq0$, we set
\[\Phi(a,b) := \frac{n(a)a^c+b^cab^c}{(n(a)-n(b))^2+(t(b^ca))^2}\,.\]
\end{definition}

We point out that the definition is well-posed because $\oo$ is a compatible $^*$-algebra. Notice that swapping $a$ and $b$ changes the numerator of $\Phi(a,b)$ but not its denominator. Moreover, $\Phi(a^c,b^c)=\Phi(a,b)^c$, $\Phi(a,0)=a^{-1}$ and $\Phi(0,b)=0$. We are now ready for the announced representation.

\begin{theorem}\label{thm:reciprocalrepresentation}
Let $f \in \mc{S}(\OO)$ and suppose that $\OO' := \OO \setminus V(N(f))$ is not empty. For all $x \in \OO' \setminus \rr$ it holds
\begin{eqnarray}\label{decomposedreciprocal}
f^{-\punto}(x) = \underbrace
{\Phi\big(\vs f(x),|\im(x)|\,f'_s(x)\big)
\phantom{\frac{}{|}}
}_{\vs{(f^{-\punto})}(x)}
\underbrace
{-\ \frac{\im(x)}{|\im(x)|}\ \Phi\big(|\im(x)|\,f'_s(x),\vs f(x)\big)}
_{\im(x)(f^{-\punto})'_s (x)}\,,
\end{eqnarray}
For all $x \in \OO' \cap (\rr \cup V(f'_s))$, $f^{-\punto}(x)=\vs f(x)^{-1}=f(x)^{-1}$. Moreover, $V((f^{-\punto})'_s)=V(f'_s)\cap\OO'$.
\end{theorem}

\begin{proof}
Pick any sphere $\alpha+\beta\s\subseteq\OO'$ (with $\alpha,\beta\in\rr$, $\beta\neq 0$). Let $a_1:=\vs f(\alpha+\beta i)$ and $a_2:=\beta f'_s (\alpha+\beta i)$ so that $f(\alpha+\beta J)=a_1+Ja_2$ for all $J\in\s$. Consider the slice function 
\[g(\alpha+\beta J):=\Phi(a_1,a_2)-J\Phi(a_2,a_1)\]
on the same sphere. Then
\[(g \cdot f)(\alpha+\beta J)=\Phi(a_1,a_2)a_1+\Phi(a_2,a_1)a_2+J(\Phi(a_1,a_2)a_2-\Phi(a_2,a_1)a_1)\,.\]
Thanks to the $^*$-Artin theorem,
\begin{align*}
\Phi(a_1,a_2)a_1+\Phi(a_2,a_1)a_2&= \frac{n(a_1)^2+(a_2^ca_1)^2+n(a_2)^2+(a_1^ca_2)^2}{(n(a_1)-n(a_2))^2+(t(a_2^ca_1))^2}\\
&= \frac{(n(a_1)-n(a_2))^2+(a_2^ca_1+a_1^ca_2)^2}{(n(a_1)-n(a_2))^2+(t(a_2^ca_1))^2}\\
&=1
\end{align*}
and
\begin{align*}
\Phi(a_1,a_2)a_2-\Phi(a_2,a_1)a_1&= \frac{n(a_1)a_1^ca_2+a_2^ca_1n(a_2)-n(a_2)a_2^ca_1-a_1^ca_2n(a_1)}{(n(a_1)-n(a_2))^2+(t(a_2^ca_1))^2}\\
&= 0\,.
\end{align*}
Thus, $g\cdot f_{|_{\alpha+\beta\s}}\equiv1$. Since proposition~\ref{reciprocal} guarantees the existence of $f^{-\punto}\in\mc{S}(\OO')$, we conclude that
\[f^{-\punto}(\alpha+\beta J) = g(\alpha+\beta J) = \Phi(a_1,a_2)-J\Phi(a_2,a_1)\]
for all $J\in\s$, which is our first statement.

The second statement follows, for the case $x \in \OO' \cap \rr$, from the fact that 
\[1=(f\cdot f^{-\punto})(x)=f(x)\,f^{-\punto}(x)=\vs f(x)\,f^{-\punto}(x)\,.\]
For all $x \in (\OO'\setminus\rr) \cap V(f'_s)$, it follows from the fact that $a_2=0$ implies, for all $J\in\s$, the equality $f^{-\punto}(\alpha+\beta J) = \Phi(a_1,0)-J\Phi(0,a_1)= a_1^{-1}$ along with the equality $f(\alpha+\beta J)=a_1=\vs f(\alpha+\beta J)$. 

By the same argument we derive the inclusion $V(f'_s)\cap\OO' \subseteq V((f^{-\punto})'_s)$. Since $f=(f^{-\punto})^{-\punto}$, the inclusion $V((f^{-\punto})'_s) \subseteq V(f'_s)$ also holds, whence the third statement.
\end{proof}

We will now study the relation between the values $f^{-\punto}(x)$ and the values $f(x)^{-1}$.

\begin{theorem}\label{thm:reciprocalformula}
Let $f \in \mc{S}(\OO)$ and suppose that $\OO':=\OO\setminus V(N(f))\neq\emptyset$. For all $x \in \OO' \cap (\rr \cup V(f'_s))$ it holds $f^{-\punto}(x)=f(x)^{-1}=f(y)^{-1}$ for each $y\in\s_x$. If $\OO'':=\OO'\setminus(\rr \cup V(f'_s))\neq\emptyset$, then for all $x\in\OO''$
\begin{equation}\label{eq:reciprocal}
f^{-\punto}(x) = f(T_f(x))^{-1}
\end{equation}
where
\begin{equation}\label{eq:transformation}
T_f(x):=(f^c(x)^{-1}((xf^c(x))f'_s(x)))f'_s(x)^{-1}\,.
\end{equation}
$T_f$ is a bijective self-map of $\OO''$. For all $y\in\OO''$, the restriction $(T_f)_{|_{\s_y}}$ is a conformal transformation of the sphere $\s_y$. If $f_{|_{\OO''}}\in\mc{S}^0(\OO'')$, then $T_f$ is a homeomorphism.

If $\OO'$ is open and $f_{|_{\OO'}}\in\mc{S}^\omega(\OO')\supseteq\mc{SR}(\OO')$ then $f'_s$ extends to $\OO'$ in real analytic fashion. Let us denote the zero set of the extension as $V$ and set $\widehat\OO:=\OO'\setminus V$. Then formula~\eqref{eq:transformation} defines a real analytic diffeomorphism $T_f$ of $\widehat\OO$ onto itself, fulfilling equality~\eqref{eq:reciprocal} for all $x\in\widehat\OO$.

Finally, in all cases described $T_{f^{-\punto}}$ is the inverse map to $T_f$.
\end{theorem}

\begin{proof}
The first statement for $\OO' \cap (\rr \cup V(f'_s))$ is a direct consequence of theorem~\ref{thm:reciprocalrepresentation}.

Let us prove the second statement concerning $\OO''$.
For $\alpha,\beta\in\rr$ such that $\alpha+\beta\s\subseteq\OO''$, we know that $f(\alpha+\beta I)=a_1+Ia_2$ for all $I\in\s$, with $a_1:=\vs f(\alpha+\beta i)$ and $a_2:=\beta f'_s (\alpha+\beta i)$, and we know that $a_2\neq0$. At $x=\alpha+\beta I$ we have that
\begin{equation}\label{eq:technical}
T_f(x)=\alpha+\beta J,\quad J=(f^c(x)^{-1}((If^c(x))a_2))a_2^{-1}\,.
\end{equation}
Now, $J \in \s$. Indeed, $n(J)=n(I)=1$ because the norm function $n$ is multiplicative. Moreover, since the trace function $t$ vanishes on all associators and commutators,
\[t(J)=t((f^c(x)^{-1}If^c(x))(a_2a_2^{-1}))=t(f^c(x)^{-1}If^c(x))=t(I)=0\,.\]
Consequently,
\begin{align*}
f(T_f(x))^{-1} &= (a_1+Ja_2)^{-1}\\
&= \left(a_1+f^c(x)^{-1}((If^c(x))a_2)\right)^{-1}\\
&= \left(f^c(x)a_1+(If^c(x))a_2\right)^{-1} f^c(x)\,.
\end{align*}
This quantity coincides with $f^{-\punto}(x)=(N(f)(x))^{-1}f^c(x)$ if, and only if,
\[f^c(x)a_1+(If^c(x))a_2 = N(f)(x)\,.\]
The last equality is equivalent to each of the following equalities:
\[(a_1^c+Ia_2^c)a_1 + (Ia_1^c-a_2^c)a_2 = n(a_1)-n(a_2) + I t(a_2^ca_1)\,,\]
\[(Ia_2^c)a_1 + (Ia_1^c)a_2 =  I (a_2^ca_1) + I (a_1^ca_2) \,,\]
\[(I,a_2^c,a_1) = -(I,a_1^c,a_2) \,,\]
\[-(I,a_2,a_1) = (I,a_1,a_2)\,,\]
where we have taken into account formulas~\eqref{decomposedconjugate} and \eqref{decomposednormal} and the fact that $t(a_2),t(a_1)$ are elements of the nucleus of $\oo$. The last equality is true by the alternating property of $\oo$. This proves the second statement concerning $\OO''$.

Now let us fix $\alpha+\beta\s\subseteq\OO''$ and prove that $(T_f)_{|_{\alpha+\beta\s}}$ is a conformal transformation of $\alpha+\beta\s$. By formula \eqref{eq:technical}, $T_f(\alpha+\beta\s)\subseteq\alpha+\beta\s$. According to theorem~\ref{thm:reciprocalrepresentation}, equality~\eqref{eq:reciprocal} can be rewritten for $x=\alpha+\beta I$ and $T_f(x)=\alpha+\beta J$ as
\[\Phi(a_1,a_2)-I\Phi(a_2,a_1)=(a_1+Ja_2)^{-1}\,\]
whence
\[J=-a_1a_2^{-1}+(\Phi(a_1,a_2)-I\Phi(a_2,a_1))^{-1}a_2^{-1}\,.\]
All affine transformations of $\oo$ are conformal (see~\cite[\S4.6, p.205]{libroward}) and the map $\rho(w)=w^{-1}=|w|^{-2}w^c$ is conformal on $\oo\setminus\{0\}$, as it is the composition between the reflection $w\mapsto w^c$ and the inversion in the unit sphere of $\rr^8$ centred at $0$. Thus, $(T_f)_{|_{\alpha+\beta\s}}$ is a conformal transformation of $\alpha+\beta\s$, as desired.

We conclude that $T_f$ is a bijective self-map of $\OO''$ because $\OO''$ is a disjoint union of spheres $\s_y$ (for appropriate $y\in\OO''$), each mapped bijectively into itself by $T_f$. 

Now let us prove that $T_{f^{-\punto}}$ is the inverse map to $T_f:\OO''\to\OO''$. We have, for $f^{-\punto}\in\mc{S}(\OO')$, that $V(N(f^{-\punto}))=V(N(f)^{-\punto})=\emptyset$. Moreover,
\[\OO'\setminus(\rr \cup V((f^{-\punto})'_s)))=\OO''\]
by theorem~\ref{thm:reciprocalrepresentation}. Thus, $T_{f^{-\punto}}$ is a bijective self-map of $\OO''$ mapping $\s_y$ into itself for all $y\in\OO''$. By applying formula~\eqref{eq:reciprocal} twice, we get that for all $x\in\OO''$
\[f(x)= (f^{-\punto})^{-\punto}(x)= f^{-\punto}(T_{f^{-\punto}}(x))^{-1} = (f(T_f(T_{f^{-\punto}}(x)))^{-1})^{-1} = f(T_f(T_{f^{-\punto}}(x)))\,.\]
Since for each $y\in\OO''$ the composition $T_f \circ T_{f^{-\punto}}$ maps $\s_y$ into itself and $f_{|_{\s_y}}$ is an affine transformation of $\s_y$ into another sphere $a_1 + \s a_2$, we conclude that $T_f \circ T_{f^{-\punto}}(x)=x$ for all $x \in \OO''$.

To conclude the proof, let us consider the case when some regularity is assumed for $f$ and let us apply~\cite[proposition 7]{perotti}.

We first deal with the case when $f\in\mc{S}^0(\OO'')$. Then $f^c,f'_s:\OO''\to\oo\setminus\{0\}$ are continuous, whence $T_f$ is continuous in $\OO''$. Because $f^{-\punto}$ is also an element of $\mc{S}^0(\OO'')$, the inverse transformation $T_{f^{-\punto}}$ is continuous, too, and $T_f$ is a homeomorphism.

Secondly, we treat the case when $\OO'$ is open and $f\in\mc{S}^\omega(\OO')$. In this case, $f^c:\OO'\to\oo\setminus\{0\}$ is real analytic and $f'_s$ extends to a real analytic function $\OO'\to\oo$. If $V$ is its zero set and $\widehat\OO:=\OO'\setminus V$ then $T_f$ extends to a real analytic map on $\widehat\OO$ by the same formula~\eqref{eq:transformation}. For all $x\in\widehat\OO\setminus\OO''$, we observe that $x$ belongs to the nucleus $\rr$ of $\oo$ so that $T_f(x)=x$. This implies both that equality~\eqref{eq:reciprocal} is still fulfilled (thanks to the first statement) and that $T_f\big(\widehat\OO\big)=\widehat\OO$. The inverse map of $T_f:\widehat\OO\to\widehat\OO$ is the analogous real analytic extension of $T_{f^{-\punto}}$ to $\widehat\OO$. In particular, $T_f:\widehat\OO\to\widehat\OO$ is a real analytic diffeomorphism.
\end{proof}

The study conducted for quaternionic slice regular functions in~\cite[theorem 5.4]{poli} and in~\cite[proposition 5.2]{zerosopen} is consistent with the previous theorem:

\begin{remark}\label{rmk:tfextends}
Formula~\eqref{eq:transformation} reduces to
\begin{equation}\label{eq:transformationspecial}
T_f(x) = f^c(x)^{-1}xf^c(x)
\end{equation}
whenever $(x,f^c(x),f'_s(x))=0$. If this associator vanishes for all $x\in\OO'$ and $f'_s\not\equiv0$, then the previous formula extends $T_f$ to a bijective self-map of $\OO'$ (a homeomorphism if $f\in\mc{S}^0(\OO')$, a real analytic diffeomorphism if $\OO'$ is open and $f\in\mc{S}^\omega(\OO')$) with inverse 
\[T_f^{-1}(x)=T_{f^{-\punto}}(x)=T_{f^c}(x)=f(x)^{-1}xf(x)\,.\]
\end{remark}

On the other hand, in the octonionic case equality~\eqref{eq:transformationspecial} does not always hold true:

\begin{example}\label{ex:ell+2xi}
Consider the octonionic polynomial $h(x)=\ell+2xi$. By direct computation, $h'_s\equiv2i$ and $h^c(x)=-\ell-2xi$; in particular, $h^c(j)=2k-\ell$. Thus,
\[T_h(j)=-((2k-\ell)^{-1}((j(2k-\ell))i))i=\frac{-3j+4\ell i}5\]
while
\[h^c(j)^{-1}jh^c(j)=(2k-\ell)^{-1}j(2k-\ell)=-j\,.\]
We point out that
\[h(h^c(j)^{-1}jh^c(j))^{-1}=h(-j)^{-1}=(\ell+2k)^{-1}\neq\frac{\ell-2k}3=h^{-\punto}(j)\,,\]
where we have taken into account that $h^{-\punto}(x)=N(h)(x)^{-1}h^c(x)=-(1+4x^2)^{-1}(\ell+2xi)$ in $\oo\setminus\frac{1}{2}\s$.
\end{example}

The previous example shows that~\cite[Formula (5.2)]{wang} is only true under additional assumptions, such as those of Remark~\ref{rmk:tfextends}.

\begin{remark}
Formula~\eqref{eq:transformation} reduces to $T_f(x)=x$ whenever
\begin{itemize}
\item $f^c(x)\in\rr$; or
\item $x$ belongs to a slice $\cc_J$ that is preserved by $f$.
\end{itemize}
The thesis follows by direct computation in both cases. In the second case, we take into account the fact that $x,f^c(x)$ and $f'_s(x)$ all belong to the commutative subalgebra $\cc_J$.
\end{remark}

In general, $T_f$ does not always admit a natural extension to $\OO'$. Let us begin with a general remark and then provide some examples.

\begin{remark}
Let $f \in \mc{S}(\OO)$, set $\OO':=\OO\setminus V(N(f))$ and let $y\in \OO' \cap V(f'_s)$. If the set $\OO'':=\OO'\setminus(\rr \cup V(f'_s))\neq\emptyset$ includes a subset $U$ (whose closure includes $y$) such that $\lim_{U\ni x \to y}f^c(x)=a$ and $\lim_{U\ni x \to y}\frac{f'_s(x)}{|f'_s(x)|}=u$, then
\[\lim_{U\ni x \to y}T_f(x)= (a^{-1}((ya)u))u^c\,.\]
\end{remark}

We are now ready to provide an example where $T_f$ admits an extension to $\OO'$, though not through formula~\eqref{eq:transformationspecial}, and an example where it does not. In both examples, we will use the Leibniz rule \eqref{decomposedproduct} for spherical derivatives and the fact that for $\Delta(x) = x^2+1$ we have $\Delta'_s(\alpha+\beta I)=2\alpha$, $\vs\Delta(\alpha+\beta I) = \alpha^2-\beta^2+1$ for all $\alpha,\beta\in\rr$.

\begin{example}
Consider the octonionic polynomial $f(x)=-i+(x^2+1)j$. By direct computation,
\[f'_s(\alpha+\beta I)\ =\ \Delta'_s(\alpha+\beta I) j\ =\ 2 \alpha j\,.\]
The zero set of the extension of $f'_s$ to $\oo$ is $\im\oo$. Thus, $T_f$ extends to a real analytic transformation of $\OO'\setminus{\im\oo}$, where $\OO':=\oo\setminus V(N(f))$. Now, $\frac{f'_s(x)}{|f'_s(x)|} = j$ if $\re(x)>0$ and $\frac{f'_s(x)}{|f'_s(x)|} = -j$ if $\re(x)<0$. Thus, the transformation $T_f$ can be analytically extended to $\OO'$ by setting 
\[T_f(x)=-(f^c(x)^{-1}((xf^c(x))j))j\,.\]
We observe that the last expression coincides with $f^c(x)^{-1}xf^c(x)$ at all $x\in\hh$ but not at $x=\ell$. Indeed, $f^c(\ell)=i$ and $T_f(\ell)=(i((\ell i)j))j=-(i(\ell k))j=-(\ell j)j=\ell$, while $f^c(\ell)^{-1} \ell f^c(\ell)=-i\ell i = -\ell$.
\end{example}

\begin{example}\label{ex:tfdoesnotextend}
Consider the octonionic polynomial $f(x)=-i+(x^2+1)(j+x\ell)$. By direct computation,
\[f'_s(\alpha+\beta I)\ =\ \Delta'_s(\alpha+\beta I) (j+\alpha\ell) + \vs\Delta(\alpha+\beta I)\ell\ =\ 2 \alpha j + (3\alpha^2-\beta^2+1)\ell\,,\]
so that $V(f'_s) = \s$. For all $I \in \s$ we have 
\[f^c(I)=i,\quad \lim_{x \to I}f'_s(x)=0,\quad \lim_{\rr\ni\alpha \to 0^{\pm}}\frac{f'_s(\alpha+I)}{|f'_s(\alpha+I)|}=\pm j,\quad \lim_{\rr\ni\beta \to 1^\pm}\frac{f'_s(\beta I)}{|f'_s(\beta I)|}=\mp\ell\,.\]
As a consequence,
\[\lim_{\rr\ni\alpha \to 0}T_f(\alpha+I) = (i((Ii)j))j,\quad \lim_{\rr\ni\beta \to 1}T_f(\beta I) = (i((Ii)\ell))\ell\,.\]
By the computations made in the previous example, $\lim_{\rr\ni\alpha \to 0}T_f(\alpha+\ell)=\ell$, which is distinct from $\lim_{\rr\ni\beta \to 1}T_f(\beta \ell) = (i((\ell i)\ell))\ell =  (ii)\ell =-\ell$. Thus, $T_f$ does not admit a continuous extension to $\ell$.
\end{example}

The previous examples notwithstanding, theorem~\ref{thm:reciprocalformula} has the following useful consequence.

\begin{corollary}\label{cor:image}
Let $f \in \mc{S}(\OO)$ and set $\OO':=\OO\setminus V(N(f))$. If $C$ is a circular nonempty subset of $\OO'$ then
\[f^{-\punto}(C) = \{f(x)^{-1} \,|\, x \in C\}\,.\]
\end{corollary} 

Our final considerations for this section concern the counterparts of formula~\eqref{eq:reciprocal} for the quotient or the product of two slice functions. The following result can be derived from~\cite[theorem 3.7]{gpsalgebra}, as well as theorems~\ref{thm:tame} and \ref{zerosonsphere}. It generalizes the results proven for quaternionic slice regular functions in~\cite[theorem 5.4]{poli},~\cite[proposition 8.1]{singularities} and~\cite[proposition 5.12]{advancesrevised} to all quaternionic slice functions.

\begin{theorem}
Let $f,g \in \mc{S}(\OO)$, where $\OO$ is a circular open subset of $\hh$. Then for all $x \in \OO\setminus V(f)$ it holds
\begin{equation} \label{eq:tfh}
(f \cdot g) (x) = f(x) g(T_{f^c}(x))\,.
\end{equation}
Moreover, for all $x \in \OO\setminus V(N(f))$, it holds
\begin{equation} \label{eq:tfh-punto}
(f^{-\punto} \cdot g) (x) = f(T_f(x))^{-1} g(T_f(x))\,.
\end{equation}
\end{theorem}

We remark that formula \eqref{eq:tfh-punto} is equivalent to
\begin{equation*}
(f^{-\punto} \cdot g) (x)=f^{-\punto}(x)g(T_f(x))\,.
\end{equation*}
The three formulas do not extend to the octonionic case, as proven in the next result and examples.

\begin{lemma}\label{lem:multbyconstant}
Let $f \in \mc{S}(\OO)$ and let $c$ be a constant different from $0$. For all $x \in \OO \cap (\rr \cup V(f'_s))$ it holds $(c \cdot f)(x)=cf(y)$ and $(f \cdot c)(x)=f(y)c$ for each $y\in\s_x$. For all $x\in\OO\setminus(\rr \cup V(f'_s))$ there exist unique $y,z\in\s_x$ such that 
\begin{align*}
&(c \cdot f)(x)=cf(y)\,,\\
&(f \cdot c)(x)=f(z)c\,;
\end{align*}
namely,
\begin{align*}
&y=(c^{-1}(x(cf'_s(x))))f'_s(x)^{-1}\,,\\
&z=((x(f'_s(x)c))c^{-1})f'_s(x)^{-1}\,.
\end{align*}
\end{lemma}

\begin{proof}
The first statement follows from the fact that $(c \cdot f)(x)=c\vs f(x)=cf(y)$ and $(f \cdot c)(x)=\vs f(x)c=f(y)c$ for all $x \in \OO \cap (\rr \cup V(f'_s))$ and for all $y\in\s_x$.

As for the second statement, pick any $x=\alpha+\beta J\in\OO\setminus(\rr \cup V(f'_s))$. Let $a_1:=\vs f(x)$ and $a_2:=\beta f'_s(x)\neq0$. Then $f(x)=a_1+Ja_2$ and
\begin{align*}
&(c\cdot f)(x) = ca_1+J(ca_2)\,,\\
&(f\cdot c)(x) = a_1c+J(a_2c)\,.
\end{align*}
The former formula equals $cf(\alpha+\beta H) = ca_1+c(Ha_2)$ if, and only if, $H=(c^{-1}(J(ca_2)))a_2^{-1}$. The latter formula equals $f(\alpha+\beta K)c = a_1c+(Ka_2)c$ if, and only if, $K=((J(a_2c))c^{-1})a_2^{-1}$. Since $y=\alpha+\beta H$ and $z=\alpha+\beta K$, the proof is complete.
\end{proof}

\begin{example}
Consider the octonionic polynomial $h(x)=\ell+2xi$ of example \ref{ex:ell+2xi}, which had $h'_s\equiv2i$. By the previous lemma, for any $c\in\oo\setminus\{0\}$ it holds
\[(h \cdot c)(x)=h\left(-((x(ic))c^{-1})i\right)c\,.\]
As a consequence, formula~\eqref{eq:tfh} is false for $f=h$ and for $g\equiv c$, even if we change $T_{f^c}$ to another transformation depending on $f$ and $g$. If we choose, for instance, $c=1+j$ then the point
$-((\ell(ic))c^{-1})i=\ell j$ is different from $\ell$ and $h(\ell j)=\ell+2(\ell j)i=\ell(1+2k)$ is different from $h(\ell)=\ell(1+2i)$.

Now let us show that formula~\eqref{eq:tfh-punto} is false when $f$ is a constant $c^{-1}$ (whence $f^{-\punto}\equiv c$), even if we change $T_f$ to another transformation depending only on $f$. By the previous lemma, 
\[(f^{-\punto}\cdot h)(x)=(c \cdot h)(x)=c\,h\left(-(c^{-1}(x(ci)))i\right)\,.\]
If we consider, instead of $h$, the function $g(x)=xc^{-1}$ (whence $g'_s\equiv c^{-1}$) then the previous lemma implies
\[(f^{-\punto}\cdot g)(x)=(c \cdot g)(x)=c\,g\left((c^{-1}(x(cc^{-1})))c\right)=c\,g(c^{-1}xc)\,.\]
If we choose, for instance, $c=1+j$, then the two transformations $x\mapsto-(c^{-1}(x(ci)))i$ and $x\mapsto c^{-1}xc$ are distinct: e.g.,
$-(c^{-1}(\ell(ci)))i=-\ell j$, while $c^{-1}\ell c=\ell j$.
\end{example}

\begin{example}
Consider again the octonionic polynomial $h(x)=\ell+2xi$ of example \ref{ex:ell+2xi}. Recall that $h'_s\equiv2i$ in $\oo$ and $h^{-\punto}(x)=-(1+4x^2)^{-1}(\ell+2xi)$ in $\oo\setminus\frac{1}{2}\s$. Let $p(x):=xj$ and let us show that formula~\eqref{eq:tfh-punto} does not hold for $f=h$ and $g=p$, even if we change $T_f$ to another transformation mapping each $\s_x$ into itself. By direct computation,
\[(h^{-\punto}\cdot p)(x)=-(1+4x^2)^{-1}x(\ell j+2xk)\]
takes the value $-k$ at $x=\ell i$, which is a point of $\s$. Now we can show that $h(J)^{-1}p(J)$ never takes the value $-k$ for $J\in\s$. Indeed, the squared modulus
\[|h(J)^{-1}p(J)|^2=(5-4\langle\ell i,J\rangle)^{-1}\]
equals $1$ if, and only if, $J=\ell i$. But $h(\ell i)^{-1}p(\ell i)=k$ is different from $-k$.
\end{example}


\section{Openness of slice regular functions}\label{sec:open}

In this section, we will state and prove the counterparts of the maximum modulus principle, the minimum modulus principle and the open mapping theorem for slice regular functions. Our statements subsume those proven in~\cite{advances,open,zerosopen,altavillawithoutreal,wang}. In the quaternionic case, a completely different approach will be adopted in~\cite{gporientation}.

Let us start with the first of these results, proven in~\cite{advances,open,zerosopen,wang} for slice domains and in~\cite{altavillawithoutreal} for quaternionic product domains.

\begin{theorem}[Maximum modulus principle]\label{thm:maximum}
Let $f\in\mc{SR}(\OO)$ and suppose $|f|$ has a local maximum point $x_0\in\OO$.
\begin{enumerate}
\item If $\OO$ is a slice domain then $f$ is constant.
\item If $\OO$ is a product domain and $x_0\in\cc_J^+$ then $f$ is constant in $\OO_J^+$.
\end{enumerate}
\end{theorem}

\begin{proof}
Suppose $\OO$ to be either a product domain or a slice domain. That is, $\OO=\OO_D$ where either: $D$ intersects the real line $\rr$, is connected and preserved by complex conjugation; or $D$ does not intersect $\rr$ and has two connected components switched by complex conjugation.

If $x_0\in\cc_J$, consider an orthonormal splitting basis 
\[\{1,J,J_1,JJ_1,J_2,JJ_2,J_3,JJ_3\}\]
for $\oo$ and apply lemma~\ref{splitting}: there are holomorphic functions $f_n:\OO_J\to\cc_J$ such that 
\[f_{|_{\OO_J}}=\sum_{n=0}^3 f_n J_{n}\,,\]
with $J_0:=1$. Let us define a map $F=(F_0,F_1,F_2,F_3): D\to\cc^4$ by letting 
\[F_n:=\phi_J^{-1}\circ f_n \circ \phi_J\,,\]
where $\phi_J^{-1}$ denotes the inverse of the bijection $\phi_J: D \to \OO_J,\ \alpha+i\beta\mapsto\alpha+J\beta$. 
The Euclidean norm $\Vert F(z)\Vert$ equals $|f(\phi_J(z))|$, whence $\Vert F\Vert$ has a local maximum point $z_0:=\phi_J^{-1}(x_0)\in D$. Since $F$ is holomorphic, it follows from the maximum modulus principle for holomorphic complex maps~\cite[theorem 2.8.3]{klimek} that $F$ is constant in the connected component of $D$ that includes $z_0$. As a consequence, $f$ is constant in the connected component of $\OO_J$ that includes $x_0$.

If $\OO$ is a product domain and $x_0\in\cc_J^+$ then $f$ is constant in $\OO_J^+$. If, on the other hand, $\OO$ is a slice domain then $f$ is constant in $\OO_J$, whence in $\OO$. 
\end{proof}

In the case of a product domain, a function that is constant on a half-slice $\OO_J^+$ may well have a local maximum point.

\begin{example}
Let $f\in\mc{SR}(\oo\setminus\rr)$ be defined by the formula
\[g(x)=3i+x\cdot f(x) = 3i+x f(x),\quad f(x)=1+\frac{\im(x)}{|\im(x)|}i\]
(using the function of Example~\ref{1fin}). In particular, $g_{|_{\cc_i^+}}\equiv3i$ and $|g|_{|_{\cc_i^+}}\equiv3$. We can see that $i$ is a local maximum point for $|g|$ as follows. First, we observe that
\[|g(x)|^2-3^2=|x|^2|f(x)|^2+6\langle i,xf(x)\rangle\,.\]
If $x\in\cc_J^+$ then $f(x)=1+Ji$, $|f(x)|^2=2-2\langle i,J \rangle$ and
\begin{align*}
\langle i,xf(x)\rangle&=\langle i,x\rangle+\left\langle i,xJi\right\rangle=\langle i,\im(x)\rangle+\left\langle 1,xJ\right\rangle=\langle i,\im(x)\rangle-|\im(x)|\\
&=|\im(x)|(\langle i,J\rangle-1)=-1/2\,|\im(x)|\,|f(x)|^2\,.
\end{align*}
Thus,
\[|g(x)|^2-3^2=(|x|^2-3|\im(x)|)|f(x)|^2\]
If we consider the product domain $U:=\{x\in\oo : |x|^2<3|\im(x)|\}$, which includes $i$, then
\[|g(i)|=3=\max_U|g|\,.\]
We point out that the same is true if we replace $i$ with any $x_0\in U_i^+$, while for $V:=\oo\setminus\overline{U}$ and for all $y_0\in V_i^+$ it holds $|g(y_0)|=3=\min_V|g|$.
\end{example}

Before turning towards the minimum modulus principle, we prove a technical lemma (cf.~\cite[proposition 6.13]{QSFCalculus} for the quaternionic case).

\begin{lemma}\label{lemma:maxmin}
Let $f\in\mc{S}(\OO)$. Choose $y=\alpha+J\beta\in \OO$ (with $\alpha,\beta\in\rr,\beta>0,J\in\s$) and let
\[v:=\vs f(y) f'_s(y)^c\,.\]
\begin{enumerate}
\item If $v\in\rr$ then $|f|_{|_{\s_y}}$ is constant.
\item Suppose $v\not\in\rr$ and set $I:=\frac{\im(v)}{|\im(v)|}$. Then the function $|f|_{|_{\s_y}}$ attains its maximum at $\alpha+\beta I$ and its minimum at $\alpha-\beta I$. In particular, the maximum and minimum of $|f|_{|_{\s_y}}$ are attained at points belonging to the subalgebra $A_{f,y}$ generated by $\vs f(y)$ and $f'_s(y)$. Moreover, $|f|_{|_{\s_y}}$ has no other local extremum.
\end{enumerate}
Thus, if $f(y)=0$ then either $f_{|_{\s_y}}\equiv0$ or $y$ is the unique local minimum point of $|f|_{|_{\s_y}}$.
\end{lemma}

\begin{proof}
For $x\in\s_y$ it holds 
\[f(x)=\vs f(x)+\im(x)f'_s(x)=\vs f(y)+\im(x)f'_s(y)\,,\]
whence
\[|f(x)|^2 = |\vs f(y)|^2+|f'_s(y)|^2+2\langle\vs f(y),\im(x)f'_s(y)\rangle =  |\vs f(y)|^2+|f'_s(y)|^2+2\langle v,\im(x)\rangle\,.\]
If $\im(v)=0$ then $\langle v,\im(x)\rangle=0$ and $|f(x)|^2$ is constant in $\s_y$. Otherwise, $|f(x)|^2$ is maximal (respectively, minimal) when $\im(x)$ is a rescaling of $\im(v)$ with a positive (respectively, negative) scale factor. Moreover, it does not admit any other local extremum.
\end{proof}

We are now ready for the minimum modulus principle. In the quaternionic case, separate results had been proven in~\cite{advances,open,zerosopen,altavillawithoutreal}. In the octonionic case,~\cite{wang} considered only the case of a slice regular function whose modulus has a local minimum point in $\rr$. For $f\in\mc{SR}(\OO)$, after restricting $f$ to $\OO':=\OO \setminus V(N(f))$, we will deal with the points of $\OO'':=\OO'\setminus (\rr\cup V(f'_s))$ by means of the transformation $T_f:\OO''\to\OO''$ defined in theorem \ref{thm:reciprocalformula}. The points in $\rr$ and the interior points of $V(f'_s)$ will be even easier to deal with, while any point of the boundary $\partial\,V(f'_s)$ (defined as the closure minus the interior of the set, as usual) will require an extra assumption.

\begin{theorem}\label{thm:minimum}
Let $f\in\mc{SR}(\OO)$. Suppose $x_0\in\OO$ to be a local minimum point for $|f|$, but not a zero of $f$. In case $x_0\in\partial\,V(f'_s)$, take one of the following additional assumptions:
\begin{enumerate}
\item[(a)] there exists a circular neighbourhood $C$ of $x_0$ such that $|f(x_0)|=\min_C|f|$; or
\item[(b)] there exist $w_0\in\s_{x_0}$ and an open neighbourhood $H$ of $w_0$ in $\OO$ such that $T_f$ continuously extends to $H$ and the extension maps $w_0$ to $x_0$.
\end{enumerate}
If $\OO$ is a slice domain, then $f$ is constant. If $\OO$ is a product domain then there exists $J\in \s$ such that $f_{|_{\OO_J^+}}\equiv f(x_0)$.
\end{theorem}

\begin{proof}
Since $f(x_0)\neq0$, lemma~\ref{lemma:maxmin} tells us that $f$ (whence $N(f)$) has no zeros in $\s_{x_0}$. In other words, $\s_{x_0}$ is included in the domain $\OO':=\OO \setminus V(N(f))$ of the reciprocal $f^{-\punto}$. Let $U$ be an open neighbourhood of $x_0$ in $\OO'$ such that 
\[|f(x_0)|=\min_U|f|\,.\]
We consider the set
\[K:=(U\cap\rr)\cup\{y\in\OO'\setminus\rr\ |\ \vs f(y) f'_s(y)^c\in\rr,\ \s_y\cap U\neq\emptyset\}\,,\]
which includes $U\cap\rr$ and $U\cap V(f'_s)$. Thus, $U\setminus K$ is included in $\OO'':=\OO'\setminus (\rr\cup V(f'_s))$ and we can set $W:=T_f^{-1}(U\setminus K)$.\\
\noindent{\bf Claim 1.} {\em If $x_0\not\in K$, then the equality
\[|f^{-\punto}(y_0)|=|f(x_0)|^{-1}=\max_{W\cup K}|f^{-\punto}|\]
holds for $y_0=T_f^{-1}(x_0)$. If $x_0\in K$ then it holds for all $y_0\in\s_{x_0}$.}\\
{\bf Proof.} We apply theorem~\ref{thm:reciprocalformula}, corollary~\ref{cor:image} and lemma~\ref{lemma:maxmin} repeatedly. We first observe that
\[\sup_{W\cup K} |f^{-\punto}| = \sup_{U\cup K} |f|^{-1} = \sup_{U} |f|^{-1} = |f(x_0)|^{-1}\,.\]
Moreover, if $x_0\in U\setminus K\subseteq\OO''$ then $f^{-\punto}(y_0)=f(x_0)^{-1}$ for $y_0=T_f^{-1}(x_0)$. If $x_0\in K\setminus (\rr\cup V(f'_s))$ then for all $y_0\in\s_{x_0}$ it holds $f^{-\punto}(y_0)=f(T_f(y_0))^{-1}$ and $|f(T_f(y_0))|=|f(x_0)|$. If $x_0\in \rr\cup V(f'_s)$ then $f^{-\punto}(y_0)=f(x_0)^{-1}$ for all $y_0\in\s_{x_0}$. {\tiny$\blacksquare$}\\
{\bf Claim 2.} {\em If $x_0\in\OO''$, then $W\cup K$ is a neighbourhood of $T_f^{-1}(x_0)$. If $x_0$ is a real point or an interior point of $V(f'_s)$ then $W\cup K$ is a neighbourhood of $\s_{x_0}$. In case {\it (a)}, if we replace $U$ with $C$, then $W\cup K=T_f^{-1}(C\setminus K)\cup K$ is a neighbourhood of $\s_{x_0}$. In case {\it (b)}, $W\cup K$ is a neighbourhood of $w_0$.}\\
{\bf Proof.} If $x_0\in\OO''$, then $T_f^{-1}(U\cap\OO'')$ is an open neighbourhood of $T_f^{-1}(x_0)$ included in $W\cup K$. If $x_0\in\rr$ then $U$ includes a circular open neighbourhood of $x_0$, which is also included in $W\cup K$. If $x_0$ is an interior point of $V(f'_s)$ then $K$ is a circular neighbourhood of $\s_{x_0}$. In case {\it (a)}, if $U$ is the circular open neighbourhood $C$ of $x_0$ then $W\cup K=T_f^{-1}(C\setminus K)\cup K=C$. Now suppose {\it (b)} holds, so that there exist $w_0\in\s_{x_0}$ and an open neighbourhood $H$ of $w_0$ in $\OO'$ such that $T_f$ extends to a continuous map $T:\OO''\cup H\to \OO'$ with $T(w_0)=x_0$. Let us consider $T^{-1}(U)$, which is an open neighbourhood of $w_0$, and let us show that $T^{-1}(U)\subseteq W\cup K$. In fact, for all $y\in T^{-1}(U)\setminus K\subseteq T^{-1}(U)\cap\OO''$ it holds $T(y)\in U$ and $T(y)=T_f(y)\in\s_y$, whence $T_f(y)\in U\setminus K$. Thus, $y\in W=T_f^{-1}(U\setminus K)$, as desired. {\tiny$\blacksquare$}

As a consequence, we can apply the maximum modulus principle~\ref{thm:maximum} to $f^{-\punto}$ at some point $y_0\in\s_{x_0}$.
If $\OO$ (whence $\OO'$) is a slice domain, we conclude that $f^{-\punto}$ is constant in $\OO'$. Thus, $f$ is constant in $\OO'$, whence in $\OO$.
If $\OO$ is a product domain, we reason as follows.

\begin{itemize}
\item The function $f^{-\punto}$ is constant in the half-slice $\OO'^+_I$ through $y_0$. Moreover, the point $y_0$ (whence $I$) can be chosen so that the constant is $f(x_0)^{-1}$. With this choice, every $y \in W_I^+:=W\cap\cc_I^+$ is still a local maximum for $|f^{-\punto}|$ with $f^{-\punto}(y)=f(x_0)^{-1}$ and every $x$ in $S:=T_f(W_I^+)$ is a local minimum for $|f|$, with $f(x)=f(x_0)$. Moreover, in $K$ it holds $|f^{-\punto}|\equiv |f(x_0)|^{-1}$ and $|f|\equiv|f(x_0)|$.
\item Let us prove that $S$ is included in a half-plane $\cc_J^+$. By lemma~\ref{lemma:maxmin}, every $x\in S$ is included in the subalgebra $A_{f,x}$ generated by $\vs f(x)$ and $f'_s(x)$, which is associative by Artin's theorem. Thanks to Remark~\ref{rmk:tfextends},
\[T_f^{-1}(x) = f(x)^{-1}xf(x) = f(x_0)^{-1}xf(x_0)\,.\]
As a consequence,
\[T_f(y)=f(x_0)yf(x_0)^{-1}\]
for all $y\in W_I^+$. Thus, $S = T_f(W_I^+)=f(x_0)W_I^+f(x_0)^{-1}$ is included in the half-plane $\cc_J^+$ with $J:=f(x_0)If(x_0)^{-1}$.
\item If $S$ is not empty then, since $S$ is an open subset of $\OO_J^+$, it follows that $f\equiv f(x_0)$ on $\OO_J^+$.
\item If $S$ is empty then $W_I^+$ is empty. Because $W_I^+\cup K_I^+$ is a neighbourhood of $y_0$ in $\OO_I^+$, it follows that $K_I^+$ is a neighbourhood of $y_0$ in $\OO_I^+$. As a consequence, the circular set $K$ is a neighbourhood of $\s_{y_0}=\s_{x_0}$ in $\OO'$. Since $|f|\equiv|f(x_0)|$ in $K$, the point $x_0$ is also a local maximum point for $|f|$. By theorem~\ref{thm:maximum}, $f$ is constant on the half-slice containing $x_0$.\qedhere
\end{itemize}
\end{proof}

Remark~\ref{rmk:tfextends} allows us to draw the following consequence, which applies to all slice preserving regular functions and to all quaternionic slice regular functions.

\begin{corollary}[Associative minimum modulus principle]\label{cor:w-a-min-modulo}
Let $f\in\mc{SR}(\OO)$ and assume that the associators $(x,f^c(x),f'_s(x))$ vanish for all $x\in\OO\setminus\rr$. Suppose $|f|$ admits a local minimum point $x_0\in\OO$, which is not a zero of $f$. If $\OO$ is a slice domain, then $f$ is constant. If $\OO$ is a product domain then there exists $J\in \s$ such that $f_{|_{\OO_J^+}}\equiv f(x_0)$.
\end{corollary}

In the octonionic setting, when $x_0$ is a boundary point of $V(f'_s)$ but neither of the assumptions {\it (a)} and {\it (b)} of theorem~\ref{thm:minimum} holds, it may well happen that no $y_0\in\s_{x_0}$ is an interior point of the set $W\cup K$ considered in the proof.

\begin{example}
Consider again the octonionic polynomial $f(x)=-i+(x^2+1)(j+x\ell)$ of Example~\ref{ex:tfdoesnotextend}. We already saw that the zero set of $f'_s(\alpha+\beta I)=2 \alpha j + (3\alpha^2-\beta^2+1)\ell$ is $\s$. Moreover,
\begin{align*}
\vs f(\alpha+\beta I)\ &=\ -i+\vs\Delta(\alpha+\beta I) (j+\alpha\ell) -\beta^2\Delta'_s(\alpha+\beta I)\ell\\
&=-i+(\alpha^2-\beta^2+1) j+ (\alpha^3-3\alpha\beta^2+\alpha)\ell\,,
\end{align*}
whence the condition $\vs f(\alpha+\beta I)f'_s(\alpha+\beta I)^c\in\rr$ is only satisfied at $\s$.

For each $I\in\s$, we saw in Example~\ref{ex:tfdoesnotextend} that both $(i((Ii)j))j$ and $(i((Ii)\ell))\ell$ are limit points of $T_f$ at $I$. Now let $U:=B(\ell,1/2)$ and $W:=T_f^{-1}(U\setminus\s)$. We will prove by contradiction that, for all $I\in\s$, the set $W\cup\s$ is not a neighbourhood of $I$.

If $W\cup\s$ were a neighbourhood of $I$, we would have $(i((Ii)j))j, (i((Ii)\ell))\ell \in U$, whence
\[(i((Ii)j))j = \ell+\gamma,\quad (i((Ii)\ell))\ell = \ell+\delta\]
for some $\gamma,\delta \in B(0,1/2)$. By direct computation, this would imply
\[I = \ell-((i(\gamma j))j)i,\quad I = -\ell+((i(\delta\ell))\ell)i\,,\]
whence the contradiction $I \in B(\ell,1/2) \cap B(-\ell,1/2)=\emptyset$.
\end{example}

Despite such pathological phenomena, theorem~\ref{thm:minimum} allows to establish the next result.

\begin{theorem}[Open mapping theorem]\label{thm:open}
Let $f\in\mc{SR}(\OO)$.
\begin{itemize}
\item If $\OO$ is a slice domain and $f$ is not constant, then $f_{|_{\OO\setminus \overline{V(f'_s)}}}$ is an open map. Moreover, the image through $f$ of any circular open subset $U\subseteq\OO$ is open. In particular, $f(\OO)$ is open.
\item If the fibers $f^{-1}(y_0)$ are discrete for all $y_0\in f(\OO)$, then $f$ is an open map.
\end{itemize}
\end{theorem}

\begin{proof}
We first deal with the case when $\OO$ is a slice domain and $f$ is not constant. Let $U$ be an open subset of $\OO\setminus\overline{V(f'_s)}$ or a circular open subset of $\OO$. Let us prove that $f(U)$ is open; that is, for each $y_0=f(x_0)$ with $x_0\in U$, let us find a radius $\varepsilon>0$ such that the Euclidean ball $B(y_0,\varepsilon)$ is contained in $f(U)$.
\begin{itemize}
\item If $U\subseteq\OO\setminus\overline{V(f'_s)}$, then the point $x_0$ must be an isolated zero for the function $g(x):=f(x)-y_0$ in $U$. Thus, there exists a closed Euclidean ball $K:=\overline{B(x_0,r)}\subseteq U$ such that $g$ never vanishes in $K\setminus\{x_0\}$.
\item Suppose $U\not\subseteq\OO\setminus\overline{V(f'_s)}$ but $U$ is circular. For an appropriate $R>0$ we have that $K:=\{x \in \oo\ |\ \mathrm{dist}(x,\s_{x_0})\leq R\}\subseteq U$ and that $g$ never vanishes in $K\setminus\s_{x_0}$.
\end{itemize}
We claim that $\varepsilon:=\frac{1}{3}\min_{\partial K}|g|$ is the desired radius. Indeed, for all $y\in B(y_0,\varepsilon)$ and for all $x\in\partial K$ the inequality $3\varepsilon\leq |g(x)|$ implies
\[|f(x_0)-y| = |y_0-y| < \varepsilon < 2\varepsilon \leq |g(x)| - |y_0-y| \leq |f(x)-y|\,.\]
Thus, $|f(x)-y|$ admits a minimum (whence a zero by theorem~\ref{thm:minimum}) at an interior point of $K$. As a consequence, $y\in f(K)\subseteq f(U)$, as desired.

Secondly, let us deal with the case when no assumption is taken on the open domain $\OO$, but the fibers $f^{-1}(y_0)$ are assumed to be discrete for all $y_0\in f(\OO)$. It suffices to prove that $f_{|_{\OO_0}}$ is open for each connected component $\OO_0$ of $\OO$. As explained in Section~\ref{sec:preliminaries}, $\OO_0$ is either a slice domain or a product domain. Moreover, the discreteness of the fibers of $f$ guarantees that $f$ is not constant in $\OO_0$, nor on any half-slice of $\OO_0$. For any open subset $U$ of $\OO_0$ and for each $y_0=f(x_0)$ with $x_0\in U$, the point $x_0$ must be an isolated zero for the function $g(x):=f(x)-y_0$ in $U$. As in the previous case, there exists a closed Euclidean ball $K:=\overline{B(x_0,r)}\subseteq U$ such that $g$ never vanishes in $K\setminus\{x_0\}$ and we can prove along the same lines that $f(U)$ includes $B(y_0,\varepsilon)$ with $\varepsilon:=\frac{1}{3}\min_{\partial K}|g|$.
\end{proof}

For quaternions, related results had been proven in~\cite{advances,open,zerosopen,altavillawithoutreal} and more will appear in~\cite{gporientation}. For octonions, the recent work~\cite{wang} had considered the case of circular open subsets of a slice domain.

We point out that, in the quaternionic and octonionic cases, restricting $f$ to $\OO\setminus \overline{V(f'_s)}$ (or supposing the fibers are discrete) is essential in order to have an open map.

\begin{example}
The slice regular polynomial $f(x)=x^2$ has $f'_s(\alpha+\beta J)=2\alpha$ for all $\alpha,\beta\in\rr$ with $\beta\neq0$ and $J\in\s$. Thus, $f'_s$ (extended to $\oo$) has $\im\oo$ as its zero set. Consider the imaginary unit $i\in V(f'_s)$ and notice that it has distance $1$ from $\cc_J$ for all $J\in\s$ orthogonal to it, e.g., $J=k$. Thus, the Euclidean ball $B(i,1)$ does not intersect $\cc_k$. Because $f$ is slice preserving, $f(B(i,1))$ does not intersect $\cc_k\setminus\rr$. As a consequence, $f(i)=-1$ is not an interior point of $f(B(i,1))$ and $f(B(i,1))$ is not open in $\oo$.
\end{example}


\section{Singularities of slice regular functions}\label{sec:singularities}

In this section, we first recall from~\cite{gpssingularities} the construction of Laurent-type expansions and the related classification of singularities as removable, essential or as poles. We then state a characterization of each type of singularity and prove an analog of the Casorati-Weierstrass theorem for essential singularities. Finally, we study the algebra of semiregular functions, namely functions without essential singularities, proving that it is a division algebra when the domain is a slice domain. These results are new in the octonionic case. Over the quaternions, part of the characterization of essential singularities is new, while the other results had been proven in~\cite{poli,zerosopen}.

Two distinct Laurent-type expansions have been presented in~\cite{gpssingularities} for slice regular functions on an alternative $^*$-algebra. Throughout this section, we continue to focus on an algebra $A=\cc,\hh$ or $\oo$. In this situation, the sets of convergence of these expansions are balls, or shells,
\begin{align*}
\Sigma(y,R) &:= \{x \in A \, | \, \sigma(x,y)<R\},\\
\Sigma(y,R_1,R_2)&:=\{x \in A \, | \, \tau(x,y)>R_1, \sigma(x,y) <R_2\},\\
\Sto(y,r)&:=\{x \in A \, | \, \sto(x,y)<r\},\\
\Sto(y,r_1,r_2)&:=\{x \in A \, | \, r_1<\sto(x,y)<r_2\}.
\end{align*}
with respect to the distance $\sigma$ and the pseudodistances $\tau,\sto$ defined as follows on $A$:
\begin{eqnarray}
&\sigma(x,y):=& \left\{
\begin{array}{ll}
|x-y|  \mathrm{\ if\ } x,y \mathrm{\ lie\ in\ the\ same\ } \cc_J \vspace{.3em}\\
\sqrt{\left(\mr{Re}(x)-\mr{Re}(y)\right)^2 + \left(|\mr{Im}(x)| + |\mr{Im}(y)|\right)^2}  \mathrm{\ otherwise}
\end{array}
\right. \ ,\\
&& \nonumber\\
&\tau(x,y):=& \left\{
\begin{array}{l}
|x-y|  \mathrm{\ if\ } x,y \mathrm{\ lie\ in\ the\ same\ } \cc_J \vspace{.3em}\\
\sqrt{\left(\mr{Re}(x)-\mr{Re}(y)\right)^2 + \left(|\mr{Im}(x)| - |\mr{Im}(y)|\right)^2}   \mathrm{\ otherwise}
\end{array}
\right. \ ,\\
&& \nonumber\\
&\sto(x,y):=&\sqrt{|\Delta_y(x)|}
\ .
\end{eqnarray}

The expansions are based on functions such as those mentioned in Examples~\ref{ex:binomial} and~\ref{ex:DeltaInv}. The first result is~\cite[theorem 4.9]{gpssingularities}:

\begin{theorem}\label{Laurent}
Consider a slice regular function $f\in\mc{SR}(\OO)$. Suppose that $y \in A$ and $R_1,R_2 \in [0,+\infty]$ are such that $R_1<R_2$ and $\Sigma(y,R_1,R_2) \subseteq \OO$. Then there exists a (unique) sequence $\{a_n\}_{n \in \zz}$ in $A$ such that
\begin{equation}\label{laurentformula}
f(x) = \sum_{n \in \zz} (x-y)^{\punto n}\cdot a_n
\end{equation}
in $\Sigma(y,R_1,R_2)$. If, moreover, $\Sigma(y,R_2)\subseteq \OO$, then for all $n<0$ we have $a_n = 0$ and formula~\eqref{laurentformula} holds in $\Sigma(y,R_2)$.
\end{theorem}

The second result is a consequence of~\cite[remark 7.4]{gpssingularities}:

\begin{theorem} \label{Laurent-spherical-expansion} 
Let $f\in\mc{SR}(\OO)$, let $y \in \cc_J\subseteq A$ and let $r_1,r_2 \in [0,+\infty]$ with $r_1 < r_2$ such that $\Sto(y,r_1,r_2) \subseteq \OO$. Then
\begin{equation} \label{eq:spherical}
f(x)=\sum_{k \in \zz}\Delta_y^k(x)(xu_k+v_k)
\end{equation}
for all $x \in \Sto(y,r_1,r_2)$, where
\begin{eqnarray}
&u_k=(2\pi J)^{-1}\int_{\partial \Sto_J(y,r)}d\zeta \, \Delta_y(\zeta)^{-k-1} f(\zeta)\,,\\
&v_k=(2\pi J)^{-1}\int_{\partial \Sto_J(y,r)}d\zeta \, (\zeta-2\re(y))\, \Delta_y(\zeta)^{-k-1} f(\zeta).
\end{eqnarray}
If, moreover, $\Sto(y,r_2)\subseteq \OO$, then for all $k<0$ we have $u_k=0=v_k$ and formula~\eqref{eq:spherical} holds in $\Sto(y,r_2)$.
\end{theorem}

The previous results allow to adopt the following terminology.

\begin{definition}\label{def:Laurent-order}
Consider a slice regular function $\mc{SR}(\OO)$. A point $y$ is a \emph{singularity} for $f$ if there exists $R>0$ such that $\Sigma(y,0,R) \subseteq \OO$, so that theorem~\ref{Laurent} and theorem~\ref{Laurent-spherical-expansion} hold with inner radii of convergence $R_1=0=r_1$ and positive outer radii of convergence $R_2,r_2$.

In the notations of theorem~\ref{Laurent}, the point $y$ is said to be a \emph{pole} for $f$ if there exists an $m\geq0$ such that $a_{n} = 0$ for all $n<-m$; the minimum such $m$ is called the \emph{order} of the pole and denoted as $\ord_f(y)$. If $y$ is not a pole, then it is called an \emph{essential singularity} for $f$ and $\ord_f(y) := +\infty$.

In the notations of theorem~\ref{Laurent-spherical-expansion}, the \emph{spherical order} of $f$ at $\s_y$ is the smallest even natural number $2k_0$ such that $u_k=0=v_k$ for all $k<-k_0$. If no such $k_0$ exists, then we set $\ord_f(\s_y) := +\infty$.

Finally, $y$ is called a \emph{removable singularity} if $f$ extends to a slice regular function in $\mc{SR}(\widetilde \OO)$, where $\widetilde \OO$ is a circular open set containing $y$.
\end{definition}

Singularities can be characterized as follows.

\begin{theorem}\label{thm:singularities}
Let $\widetilde{\OO}$ be a circular open set, let $y \in \widetilde{\OO}\setminus\rr$ and set $\OO:=\widetilde{\OO}\setminus \s_y$. If $f\in\mc{SR}(\OO)$ then one of the following assertions holds:
\begin{enumerate}
\item Every point of $\s_y$ is a removable singularity for $f$, i.e., $f$ extends to a slice regular function on $\widetilde{\OO}$. It holds $\ord_f(\s_y)=0$ and $\ord_f(w)=0$ for all $w \in \s_y$.
 \item Every point of $\s_y$ is a non removable pole for $f$. There exists $k \in \nn \setminus \{0\}$ such that the function defined on $\OO$ by the expression
\[x \mapsto \Delta_y^k(x)f(x)\]
extends to a slice regular function $g \in \mc{SR}(\widetilde{\OO})$ that has at most one zero in $\s_y$. It holds $\ord_f(\s_y)=2k$. Moreover, $\ord_f(w)=k$ and $\lim_{\OO \ni x \to w} |f(x)| = +\infty$ for all $w$ in $\s_y$ except the possible zero of $g$, which must have order less than $k$.
\item Every point $w\in\s_y$, except at most one, is an essential singularity, i.e., $\ord_f(w)=+\infty$. It holds $\ord_f(\s_y)=+\infty$. For all neighbourhoods $U$ of $\s_y$ in $\widetilde{\OO}$ and for all $k \in \nn$,
\[\sup_{x \in U \setminus \s_y}|\Delta_y^k(x)f(x)|=+\infty.\]
\end{enumerate}
In particular, to check which is the case it suffices to check whether $\ord_f(y),\ord_f(y^c)$ are both $0$, both finite (but not both $0$) or not both finite; or, equivalently, (if $J \in \s_A$ is such that $y \in \cc_J\setminus \rr$) whether the function defined on $\OO_J$ by the expression
\[z \mapsto (z-y)^k(z-y^c)^kf(z)\]
is bounded near $y$ and $y^c$ for $k=0$, for some finite $k$ or for no $k\in\nn$.
\end{theorem}

\begin{proof}
We can apply~\cite[theorem 9.4]{gpssingularities} to $f$. Moreover, according to theorem~\ref{zerosonsphere}, the function $g$ appearing in case {\it 2} can have at most one zero. Finally, from the same theorem and from~\cite[lemma 10.7]{gpssingularities} it follows that in case {\it 3} there can be at most one point in $\s_y$ that is not an essential singularity for $f$.
\end{proof}

A similar characterization is available for real singularities, see~\cite[theorem 9.5]{gpssingularities}. It is completely analogous to the complex case.

We can now consider the analogs of meromorphic functions.

\begin{definition}\label{def:semiregular}
A function $f$ is \emph{(slice) semiregular} in a (nonempty) circular open set $\widetilde{\OO}$ if there exists a circular open subset $\OO$ of $\widetilde{\OO}$ such that $f \in \mc{SR}(\OO)$ and such that every point of $\widetilde{\OO} \setminus \OO$ is a pole (or a removable singularity) for $f$.
\end{definition}

The algebra of semiregular functions can be studied as follows.

\begin{theorem} \label{thm:semiregular}
Let $\OO$ be a a circular open set. The set $\mc{SEM}(\OO)$ of semiregular functions on $\OO$ is an alternative $^*$-algebra with respect to $+,\cdot,^c$.
\begin{itemize}
\item If $\OO$ is a slice domain, $\mc{SEM}(\OO)$ is a division algebra.
\item If $\OO$ is a product domain, then $\mc{SEM}(\OO)$ is a singular algebra, that is, it includes some element $f\not\equiv0$ with $N(f)\equiv0$. However, every element $f$ with $N(f)\not\equiv0$ admits a multiplicative inverse within the algebra.
\end{itemize}
\end{theorem}

\begin{proof}
The first statement is derived from~\cite[theorem 11.3]{gpssingularities}. When $\OO$ is a slice domain,~\cite[corollary 11.7]{gpssingularities} tells us that the nonzero tame elements of the algebra form a multiplicative Moufang loop. But in our setting all slice functions $f$ are tame by theorem~\ref{thm:tame}, whence the second statement follows. Finally, if $\OO$ is a product domain then the algebra is singular by proposition~\ref{SRnonsingular} (see also example~\ref{1fin}). However, every $f$ with $N(f)\not\equiv0$ admits a multiplicative inverse (still semiregular in $\OO$) by~\cite[theorem 11.6]{gpssingularities}, if we take into account again theorem~\ref{thm:tame}.
\end{proof}

We are now ready to prove our last result: an analog of the Casorati-Weierstrass theorem for essential singularities.

\begin{theorem} \label{thm:casorati}
Let $\widetilde{\OO}$ be a circular open set. Let $y \in \widetilde{\OO}$ and set $\OO:=\widetilde{\OO}\setminus \s_y$. Suppose $y$ to be an essential singularity for $f\in\mc{SR}(\OO)$. Then, for each neighbourhood $U$ of $\s_y$ in $\widetilde{\OO}$, the image $f(U\setminus\s_y)$ is dense in $A$.
\end{theorem}

\begin{proof}
We will prove that, if there exist a neighbourhood $U$ of $\s_y$ and a Euclidean ball $B(v,r)$ with $r>0$ such that
\[f(U\setminus\s_y)\cap B(v,r)=\emptyset\,,\]
then $y$ is not an essential singularity for $f$. We assume, without loss of generality, $U$ to be circular.

The previous equality implies that the function $g:=f-v$ maps $U\setminus\s_y$ into the complement of $B(0,r)$.
Now consider $g^{-\punto}$, which is semiregular in $\OO$. By corollary~\ref{cor:image},
\[g^{-\punto}(U\setminus\s_y) = \{g(x)^{-1}\ |\ x \in U\setminus\s_y\}\subseteq \overline{B(0,1/r)}\,,\]
whence $g^{-\punto}$ is bounded near $\s_y$. By theorem~\ref{thm:singularities}, $g^{-\punto}$ extends to a regular function on $\widetilde{\OO}$. Using theorem~\ref{thm:semiregular} twice, we conclude that $g$ and $f$ are semiregular in $\widetilde{\OO}$, so that $y$ cannot be an essential singularity for $f$.
\end{proof}

\newpage

\section*{Acknowledgements}
This work was partly supported by GNSAGA of INdAM, by the grant FIRB 2012 ``Differential Geometry and Geometric Function Theory" of the Italian Ministry of Education (MIUR) and by the INdAM project ``Hypercomplex function theory and applications''. The third author is also supported by Finanziamento Premiale FOE 2014 ``Splines for accUrate NumeRics: adaptIve models for Simulation Environments'' of MIUR.

We warmly thank the anonymous referee, whose helpful suggestions have significantly improved our presentation.




\begin{thebibliography}{24}

\bibitem{altavillawithoutreal}
A.~Altavilla.
\newblock Some properties for quaternionic slice regular functions on domains
  without real points.
\newblock {\em Complex Var. Elliptic Equ.}, 60(1):59--77, 2015.

\bibitem{baez}
J.~C. Baez.
\newblock The octonions.
\newblock {\em Bull. Amer. Math. Soc. (N.S.)}, 39(2):145--205, 2002.

\bibitem{cauchy}
F.~Colombo, G.~Gentili, and I.~Sabadini.
\newblock A {C}auchy kernel for slice regular functions.
\newblock {\em Ann. Global Anal. Geom.}, 37(4):361--378, 2010.

\bibitem{advancesrevised}
F.~Colombo, G.~Gentili, I.~Sabadini, and D.~Struppa.
\newblock Extension results for slice regular functions of a quaternionic
  variable.
\newblock {\em Adv. Math.}, 222(5):1793--1808, 2009.

\bibitem{ebbinghaus}
H.-D. Ebbinghaus, H.~Hermes, F.~Hirzebruch, M.~Koecher, K.~Mainzer,
  J.~Neukirch, A.~Prestel, and R.~Remmert.
\newblock {\em Numbers}, volume 123 of {\em Graduate Texts in Mathematics}.
\newblock Springer-Verlag, New York, 1990.
\newblock With an introduction by K. Lamotke, Translated from the second German
  edition by H. L. S. Orde, Translation edited and with a preface by J. H.
  Ewing, Readings in Mathematics.

\bibitem{zeros}
G.~Gentili and C.~Stoppato.
\newblock Zeros of regular functions and polynomials of a quaternionic
  variable.
\newblock {\em Michigan Math. J.}, 56(3):655--667, 2008.

\bibitem{open}
G.~Gentili and C.~Stoppato.
\newblock The open mapping theorem for regular quaternionic functions.
\newblock {\em Ann. Sc. Norm. Super. Pisa Cl. Sci. (5)}, VIII(4):805--815,
  2009.

\bibitem{zerosopen}
G.~Gentili and C.~Stoppato.
\newblock The zero sets of slice regular functions and the open mapping
  theorem.
\newblock In I.~Sabadini and F.~Sommen, editors, {\em Hypercomplex analysis and
  applications}, Trends in Mathematics, pages 95--107. Birkh\"auser Verlag,
  Basel, 2011.

\bibitem{librospringer}
G.~Gentili, C.~Stoppato, and D.~C. Struppa.
\newblock {\em Regular functions of a quaternionic variable}.
\newblock Springer Monographs in Mathematics. Springer, Heidelberg, 2013.

\bibitem{survey}
G.~Gentili, C.~Stoppato, D.~C. Struppa, and F.~Vlacci.
\newblock Recent developments for regular functions of a hypercomplex variable.
\newblock In I.~Sabadini, M.~V. Shapiro, and F.~Sommen, editors, {\em
  Hypercomplex analysis}, Trends Math., pages 165--185. Birkh{\"a}user Verlag,
  Basel, 2009.

\bibitem{cras}
G.~Gentili and D.~C. Struppa.
\newblock A new approach to {C}ullen-regular functions of a quaternionic
  variable.
\newblock {\em C. R. Math. Acad. Sci. Paris}, 342(10):741--744, 2006.

\bibitem{advances}
G.~Gentili and D.~C. Struppa.
\newblock A new theory of regular functions of a quaternionic variable.
\newblock {\em Adv. Math.}, 216(1):279--301, 2007.

\bibitem{rocky}
G.~Gentili and D.~C. Struppa.
\newblock Regular functions on the space of {C}ayley numbers.
\newblock {\em Rocky Mountain J. Math.}, 40(1):225--241, 2010.

\bibitem{QSFCalculus}
R.~Ghiloni, V.~Moretti, and A.~Perotti.
\newblock Continuous slice functional calculus in quaternionic {H}ilbert
  spaces.
\newblock {\em Rev. Math. Phys.}, 25(4):1350006, 83, 2013.

\bibitem{gporientation}
R.~Ghiloni and A.~Perotti.
\newblock On a class of orientation-preserving maps of $\mathbb{R}^4$.
\newblock arxiv.org/abs/1902.11227 (Submitted).

\bibitem{perotti}
R.~Ghiloni and A.~Perotti.
\newblock Slice regular functions on real alternative algebras.
\newblock {\em Adv. Math.}, 226(2):1662--1691, 2011.

\bibitem{ghiloni}
R.~Ghiloni and A.~Perotti.
\newblock Zeros of regular functions of quaternionic and octonionic variable: a
  division lemma and the camshaft effect.
\newblock {\em Ann. Mat. Pura Appl. (4)}, 190(3):539--551, 2011.

\bibitem{gpsalgebra}
R.~Ghiloni, A.~Perotti, and C.~Stoppato.
\newblock The algebra of slice functions.
\newblock {\em Trans. Amer. Math. Soc.}, 369(7):4725--4762, 2017.

\bibitem{gpssingularities}
R.~Ghiloni, A.~Perotti, and C.~Stoppato.
\newblock Singularities of slice regular functions over real alternative
  {$^*$}-algebras.
\newblock {\em Adv. Math.}, 305:1085--1130, 2017.

\bibitem{klimek}
M.~Klimek.
\newblock {\em Pluripotential theory}, volume~6 of {\em London Mathematical
  Society Monographs. New Series}.
\newblock The Clarendon Press, Oxford University Press, New York, 1991.
\newblock Oxford Science Publications.

\bibitem{poli}
C.~Stoppato.
\newblock Poles of regular quaternionic functions.
\newblock {\em Complex Var. Elliptic Equ.}, 54(11):1001--1018, 2009.

\bibitem{singularities}
C.~Stoppato.
\newblock Singularities of slice regular functions.
\newblock {\em Math. Nachr.}, 285(10):1274--1293, 2012.

\bibitem{wang}
X.~Wang.
\newblock On geometric aspects of quaternionic and octonionic slice regular
  functions.
\newblock {\em J. Geom. Anal.}, 27(4):2817--2871, 2017.

\bibitem{libroward}
J.~P. Ward.
\newblock {\em Quaternions and {C}ayley numbers}, volume 403 of {\em
  Mathematics and its Applications}.
\newblock Kluwer Academic Publishers Group, Dordrecht, 1997.
\newblock Algebra and applications.

\end{thebibliography}
\end{document}